\documentclass[10pt,reqno]{amsart}

\usepackage[a4paper,left=26mm,right=26mm,top=28mm,bottom=28mm,marginpar=25mm]{geometry}

\usepackage{amsmath}
\usepackage{amssymb}
\usepackage{amsthm}
\usepackage[latin1]{inputenc}
\usepackage{eurosym}
\usepackage[dvips]{graphics}

\usepackage{graphicx}

\usepackage{epsfig}

\usepackage{hyperref}
\usepackage{dsfont}
\usepackage{color}
\usepackage[displaymath,mathlines]{lineno}

\allowdisplaybreaks

\usepackage{hyperref}

\usepackage{colonequals}
\usepackage{ifthen}

\makeindex

\usepackage[nocompress, space]{cite}
\usepackage{bm}

\renewcommand{\div}{\operatorname{div}}

\newcommand{\Rr}{{\mathbb{R}}}

\newcommand{\Nn}{{\mathbb{N}}}

\newcommand{\Tt}{{\mathbb{T}}}

\newcommand{\Ll}{{\mathcal{L}}}

\newcommand{\Ff}{{\mathfrak{F}}}

\newcommand{\Aa}{{\mathcal{A}}}
\newcommand{\Bb}{{\mathcal{B}}}

\newcommand{\Uu}{{\mathcal{U}}}

\newcommand{\Oo}{{\mathcal{O}}}

\newcommand{\bx}{{\bf x}}

\newcommand{\bdx}{\dot{\bf x}}

\newcommand{\epsi}{\epsilon}

\def\dx{{\rm d}x}

\def\dt{{\rm d}t}

\def\leq{\leqslant}
\def\geq{\geqslant}

\numberwithin{equation}{section}

\newtheoremstyle{thmlemcorr}{10pt}{10pt}{\itshape}{}{\bfseries}{.}{10pt}{{\thmname{#1}\thmnumber{
#2}\thmnote{ (#3)}}}
\newtheoremstyle{thmlemcorr*}{10pt}{10pt}{\itshape}{}{\bfseries}{.}\newline{{\thmname{#1}\thmnumber{
#2}\thmnote{ (#3)}}}
\newtheoremstyle{defi}{10pt}{10pt}{\itshape}{}{\bfseries}{.}{10pt}{{\thmname{#1}\thmnumber{
#2}\thmnote{ (#3)}}}
\newtheoremstyle{remexample}{10pt}{10pt}{}{}{\bfseries}{.}{10pt}{{\thmname{#1}\thmnumber{
#2}\thmnote{ (#3)}}}
\newtheoremstyle{ass}{10pt}{10pt}{}{}{\bfseries}{.}{10pt}{{\thmname{#1}\thmnumber{
A#2}\thmnote{ (#3)}}}

\theoremstyle{thmlemcorr}

\newtheorem{theorem}{Theorem}
\numberwithin{theorem}{section}
\newtheorem{lemma}[theorem]{Lemma}
\newtheorem{corollary}[theorem]{Corollary}

\theoremstyle{thmlemcorr*}
\newtheorem{theorem*}{Theorem}
\newtheorem{lemma*}[theorem]{Lemma}
\newtheorem{corollary*}[theorem]{Corollary}
\newtheorem{proposition*}[theorem]{Proposition}
\newtheorem{problem*}[theorem]{Problem}
\newtheorem{conjecture*}[theorem]{Conjecture}

\theoremstyle{defi}
\newtheorem{definition}[theorem]{Definition}
\newtheorem{hyp}{Assumption}

\theoremstyle{remexample}
\newtheorem{remark}[theorem]{Remark}

\newtheorem{teo}[theorem]{Theorem}
\newtheorem{lem}[theorem]{Lemma}
\newtheorem{pro}[theorem]{Proposition}

\theoremstyle{ass}

\begin{document}

\title[Monotonicity methods in mean-field games]{An introduction to monotonicity methods\\
	 in mean-field games}

\author{Rita Ferreira}
\address[R. Ferreira]{
        King Abdullah University of Science and Technology (KAUST), CEMSE Division, Thuwal 23955-6900, Saudi Arabia.}
\email{rita.ferreira@kaust.edu.sa}
\author{Diogo Gomes}
\address[D. Gomes]{
        King Abdullah University of Science and Technology (KAUST), CEMSE Division, Thuwal 23955-6900, Saudi Arabia.}
\email{diogo.gomes@kaust.edu.sa}
\author{Teruo Tada}
\address[T. Tada]{
       King Abdullah University of Science and Technology (KAUST), CEMSE Division, Thuwal 23955-6900, Saudi Arabia.}
\email{teruo.tada@kaust.edu.sa}



\thanks{R. Ferreira, D. Gomes, and T. Tada were partially supported by baseline and start-up funds from King Abdullah University of Science and Technology (KAUST) OSR-CRG2017-3452 and OSR-CRG2021-4674.}

\begin{abstract}
This chapter examines monotonicity techniques in the theory of mean-field games (MFGs). 
Originally, monotonicity ideas were used to establish the uniqueness of solutions for MFGs. Later, 
monotonicity methods and monotone operators were further exploited to build numerical methods and to construct  weak solutions under mild assumptions. Here, after a brief discussion on the mean-field game formulation, we introduce the Minty method and regularization strategies for PDEs. These
are then used 
to  address typical stationary and time-dependent   monotone MFGs and to establish the existence of weak solutions for such MFGs.  
\end{abstract}

\maketitle
\tableofcontents

\section{Introduction}\label{int1}
MFGs model the limit of differential games with a large population as the number of agents tends to infinity. In these models, each agent is rational and optimizes a cost functional.
MFGs were motivated by questions in economics and engineering, and were introduced by Lasry and Lions in \cite{ll1, ll2, ll3} and Caines, Huang, and Malham\'{e} in \cite{Caines2, Caines1}.
Often, these games comprise a system of two equations, a Hamilton--Jacobi equation and a Fokker--Planck equation, which can be associated with a monotone operator. This structure
is key to establish the uniqueness of solutions, as used in \cite{ll2}. Here, we illustrate further applications and techniques. In particular, we address 
how to build weak solutions for
stationary and time-dependent MFGs.

A main difficulty in the MFG theory is to establish the existence of solutions. While Hamilton--Jacobi equations and Fokker--Planck equations are extremely well understood, the coupling between these two equations presents substantial difficulties. Monotone operators ideas give a unified approach to address a wide class of MFGs. This approach, which we illustrate here,  was first introduced in \cite{FG2} and was
motivated by prior works in the numerical approximation of MFGs in \cite{almulla2017two}. Subsequently,
it was further developed in 
\cite{FGT1} for stationary problems with boundary conditions,   in \cite{FeGoTa21} for time-dependent problems,     and in \cite{DRT2021Potential} for the planning problem with congestion. The purpose of the present chapter is to describe in a simplified setting the main ideas on those references.

In recent years, there has been much progress on the existence and regularity of solutions for
MFGs. Most of the available literature relies heavily on the structure of the specific problem under 
consideration and often do not generalize in a straightforward way to other problems.
A historical account of part of the theory can be found in the lectures \cite{LCDF} and is available in 
the books and surveys \cite{cardaliaguet,bensoussan, GNP, GPV, GS,carmona2018probabilistic, MR4214773}. Partial differential equation methods have played an essential role in the development of MFGs from the beginning, see  \cite{ll1, ll2}.
Subsequently, several authors have looked into core matters such as the existence and regularity of solutions. Smooth solutions  were obtained in 
\cite{GM, MR3113415, GPatVrt, PV15} for the stationary case and in  
\cite{GPM2, GPM3, Gomes2015b, Gomes2016c, CirGo20} for the time-dependent case.
These results were obtained by combining the non-linear adjoint method to obtain a priori estimates with the continuation method. 
Around the same
time, the theory of weak solutions saw substatial progress in the 
elliptic case in \cite{bocorsporr, AMFS} and in the parabolic case in 
\cite{porretta, porretta2,  cgbt,MR3691806}. 
Degenerate elliptic or parabolic problems
were examined in \cite{cgbt} and, in the particular
hypoelliptic case, were studied in \cite{DragoniFeleqi2018, MR4132070, MaMaTch2021}.
First-order MFGs present substantial challenges and were
examined in \cite{Cd2, GraCard, MR4175148}.
The theory for MFGs for problems with boundary conditions and state constraints, especially for first-order MFGs, is not yet complete.
The case where there is an invariance condition for the state space was investigated in \cite{MR4045803}.
The authors in \cite{MR3888967, MR4132067, MR4135073, CaCaCa2018} use Lagrangian methods to investigate MFGs with state constraints.
This approach relies on a weak notion of MFG equilibrium defined in a Lagrangian setting rather than with a PDE system.
In  \cite{MR3888967, MR3395203, MaMaTch2021}, the equilibria are probability measures defined on a set of admissible trajectories.
Once the existence of a relaxed MFG equilibrium is ensured, we can investigate the regularity of solutions
and give meaning to the PDE system and the related boundary conditions \cite{MR4135073,CaCaCa2018}.

While our focus here are motonote MFGs, this introduction would not be 
complete without mentioning some of the research on non-monotone MFGs. 
These were studied in particular one-dimensional cases in 
\cite{Gomes2016b}. Then, using PDE methods,
\cite{cirant3,Cirpor2021, ciranttonon} prove the existence under suitable growth assumptions 
for the problem
and in \cite{MR4064664} for the short-time problem. 
Finally, the 
short-time case or small data was examined in \cite{Ambrose} based on the analysis
of the Fourier coefficients of the solutions, a technique derived from prior works in fluid mechanics.

A model stationary MFG is the following. Agents' state is the $d$-dimensional torus, $\Tt^d$. 
The distribution of agents is given by a density $m:\Tt^d\to \Rr_0^+$, not necessarily normalized, where   \(\Rr^+_0=[0,+\infty)\). 
Each agent selects a trajectory $\bx:\Rr_0^+\to \Tt^d$ and
 seeks to minimize the discounted cost
\[
\int_0^{+\infty} e^{-t}(L(\bdx(t))-V(\bx(t)) +g(m(\bx(t))))\,\dt.
\]
The discount term $e^{-t}$ accounts for agents exiting the game at a random time
with an exponential constant rate. These exiting agents are replaced by agents that begin
their trajectories at a random point in $\Tt^d$ chosen according to a rate function $\phi\geq 0$.
The Lagrangian $L:\Rr^d\to \Rr$ is a uniformly convex function (or, alternatively, strictly convex and coercive) that encodes the cost of moving at the speed $\bdx(t)$. The potential $V:\Tt^d\to \Rr$ 
represents the spatial preferences of the agents. Finally, $g:\Rr_0^+\to \Rr$ determines congestion 
effects and is where the interaction between each agent and the population is encoded.
 Let
\begin{equation}
\label{hamlt}
H(p)=\sup_{v\in\Rr^d} (-v \cdot p -L(v)), \quad p\in\Rr^d,
\end{equation} 
be the Legendre transform of $L$. The value function,
\[
u(x)=\inf_{\bx\in W^{1,\infty}(0,+\infty)\atop \bx(0)=x}\int_0^{+\infty} e^{-t}(L(\bdx(t))-V(\bx(t)) +g(m(\bx(t))))\,\dt, 
\]
if regular enough, it
solves the Hamilton--Jacobi equation 
\[
u+H(Du) - g(m) + V =0. 
\]
Further, the optimal trajectory is given in feedback form by 
\[
\bdx=-D_pH(Du(\bx)).
\]
In this model, agents are rational and therefore the whole population 
follows the preceding ODE. This gives rise to the transport equation
\[
m-\div\big(m D_p H(Du)\big) -\phi =0 
\]
that determines $m$. Gathering the Hamilton--Jacobi and the transport equation, 
we obtain the  MFG system 
\begin{equation}\label{mp}
        \begin{cases}
                - u-H(Du) + g(m) - V =0\ \ \mbox{in }\Tt^d,\\
                m-\div\big(m D_p H(Du)\big) -\phi =0 \ \ \mbox{in }\Tt^d.
        \end{cases}
\end{equation}
For more details on the derivation of MFG models, including the time-dependent case, we refer the reader to \cite{GPV}.
The preceding system can be associated with 
the operator
\begin{equation}\label{defF1}
        F
        \begin{bmatrix}
                w \\
                v
        \end{bmatrix}
        =
        \begin{bmatrix}
                -v-H(Dv)+g(w)-V \\
                w-\div(wD_pH(D v))-\phi
        \end{bmatrix}.
\end{equation}
If $g$ is increasing and $H$ is convex, $F$ is monotone on its domain $D\subset L^2(\Tt^d; \Rr^+_0)\times L^2(\Tt^d)$. As we mentioned before, the focus of this chapter is to explore this monotonicity.

A first application of monotonicity is the uniqueness of solutions to MFGs, which is discussed in Section~\ref{mouni}. Concerning existence of solutions, we work with weak solutions. Before 
addressing MFGs, for pedgagogical reasons, we illustrate the main techniques on simpler
PDEs in 
Section~\ref{exmono}. There, we define weak solutions for monotone operators, 
discuss the Minty's method to obtain weak solutions, and illustrate 
two main technical points, regularization strategies and corresponding existence of solutions 
for the regularized problems. We then extend these ideas to address MFGs in Sections~\ref{rpMFG} and \ref{t-dep5}.

\section{Monotonicity and uniqueness}\label{mouni}
Often, the equations  comprising MFGs can be interpreted as
a monotone operator, whose notion we recall next.  Let $H$ be
a Hilbert space endowed with an inner product, \((\cdot,\cdot)\),
and let $D(\cdot)$ be the domain of an operator on \(H\). We
say that an operator, $\Ff: D(F) \subset H \to H$, is monotone
if for every $u, v\in D(\Ff)$, $\Ff$ satisfies
\begin{equation}\label{def:mono}
\big(\Ff(u)-\Ff(v), u-v\big)\geq 0.
\end{equation}
We say that  $\Ff$
is strictly monotone if for all  $u,v \in D(\Ff)\subset H$  such
that $u\not\equiv
v$,  we have
\begin{equation}\label{def:stmono}
(\Ff(u)-\Ff(v), u-v )> 0.
\end{equation} 
If \(\Ff\) is strictly monotone, then there exists at most one
solution to the problem of finding \(u\in D(\Ff)\) for which \(\Ff(u)=0\).
In fact, if \(\Ff\) is strictly monotone and   $u$, $v\in D(\Ff)$
are such that  $\Ff(u)=\Ff(v)=0$, then the left-hand side of \eqref{def:stmono}
vanishes, which can only be if \(u\equiv v\).

\begin{pro}
\label{p21}
Consider 
the operator $F$ given by \eqref{defF1} associated with the MFG in \eqref{mp}.
Assume that $H$ is convex,  $g$ is increasing,   and let $D(F)=C^2(\Tt^d;
\Rr_0^+)\times C^1(\Tt^d)$. Then, $F$ is monotone in $L^2(\Tt^d)\times L^2(\Tt^d)$. \end{pro}

\begin{proof}
Let $\langle \cdot,\cdot\rangle$ denote the inner product in  $L^2(\Tt^d)\times L^2(\Tt^d)$.
Given
$(m_1, u_1)$, $(m_2, u_2)\in D(F)$, integrating by parts
and using the assumptions on $H$ and $g$ yields
\begin{align}\label{A1mono}
                &\left\langle 
                F
                \begin{bmatrix}
                        m_1 \\
                        u_1
                \end{bmatrix}
                -
                F
                \begin{bmatrix}
                        m_2 \\
                        u_2
                \end{bmatrix}
                ,
                \begin{bmatrix}
                        m_1 \\
                        u_1
                \end{bmatrix}
                -  \begin{bmatrix}
                        m_2 \\
                        u_2
                \end{bmatrix}
                \right\rangle\notag\\&\quad
                =
                \int_{\Tt^d}
                m_1\big(H(D u_2)- H(Du_1)- D_p H(D u_1) \cdot (D u_2
                - D u_1) \big)\,\dx\notag\\
                &\qquad+\int_{\Tt^d}
                m_2\big(H(Du_1) - H(D u_2)-  D_p H(Du_2)\cdot (D u_1-D
                u_2)\big) \,\dx\notag\\
                &\qquad+\int_{\Tt^d}(m_1-m_2)(g(m_1)- g(m_2)) 
                \,\dx
                \geq 0,
\end{align}
which establishes the monotonicity of $F$. 
\end{proof}

\begin{remark}
Under the assumptions of the preceding proposition, 
the following second-order stationay MFG
\begin{equation}\label{exs1}
        \begin{cases}
                u - \Delta u + H(Du) - g(m) - V(x)=0&\mbox{in}\ \Tt^d,\\
                m - \Delta m - \div(mD_p H(Du)) - \phi(x)=0&\mbox{in}\ \Tt^d,
        \end{cases}
\end{equation}
and time-dependent MFG 
\begin{equation}\label{ext1}
        \begin{cases}
                - u_t - \Delta u + H(Du) - g(m) - V=0& \mbox{in}\ (0,T)\times\Tt^d,\\
                m_t - \Delta m - \div(mD_p H(Du)) =0& \mbox{in}\ (0,T)\times\Tt^d,\\
        \end{cases}
\end{equation}
with initial-terminal boundary conditions
\begin{equation}
        \label{ext1bc}
                m(0,\cdot)=m_0,\ u(T,\cdot)=u_T \quad  \mbox{in}\ \Tt^d,
\end{equation}
with $m_0:\Tt^d\to \Rr_0^+$, $\int_{\Tt^d} m_0=1$, and $u_T:\Tt^d \to \Rr$, share the same monotonone structure. More precisely,  
define $F_1\colon D(F_1) \subset L^2(\Tt^d)\times L^2(\Tt^d)
\to L^2(\Tt^d)\times L^2(\Tt^d)$
\begin{equation*}
        \begin{aligned}
                F_1
                \begin{bmatrix}
                        m \\
                        u
                \end{bmatrix}= 
                \begin{bmatrix}
                        -u + \Delta u - H(D u) + g(m) + V \\
                        m - \Delta m - \div(mD_pH(D u))-\phi
                \end{bmatrix}, \quad (m,u) \in D(F_1),
        \end{aligned}
\end{equation*}
 and 
$F_2\colon D(F_2)\subset L^2([0,T]\times \Tt^d)\times L^2([0,T]\times \Tt^d) \to  L^2([0,T]\times \Tt^d)\times L^2([0,T]\times \Tt^d) $ by 
\begin{equation*}
        \begin{aligned}
                F_2
                \begin{bmatrix}
                        m \\
                        u
                \end{bmatrix}= 
                \begin{bmatrix}
                        u_t + \Delta u - H(D u) + g(m) + V \\
                        m_t - \Delta m - \div(mD_pH(D u)) 
                \end{bmatrix}, \quad (m,u) \in D(F_2).
        \end{aligned}
\end{equation*}
By selecting $D(F_1)=C^2(\Tt^d; \Rr_0^+)\times C^2(\Tt^d)$ and 
\[
D(F_2)=\{m\in C^2([0, T]\times \Tt^d ;\Rr_0^+): m(x,0)=m_0\}
\times \{ u\in C^2([0, T]\times \Tt^d): u(x, T)=u_T\}
\]
then, following the steps in the proof of Proposition 
\ref{p21}, we obtain the monotonicity of $F_1$ and $F_2$ in 
$L^2(\Tt^d)\times L^2(\Tt^d)$
and
$L^2([0,T]\times \Tt^d)\times L^2([0,T]\times \Tt^d)$, respectively. 
\end{remark}

The previous operators are not strictly monotone, even if $H$ is strictly convex and $g$ is strictly increasing; the lack of strict monotonicity arises from the possibility of $m$ vanishing, thus
\eqref{A1mono} could be identically $0$ without $Du_1=Du_2$. 
Thus, to prove uniqueness of solutions, 
we either need to add further conditions on those solutions or provide an additional arguments
that give the strict positivity of $m$. 
\begin{pro}
Consider 
the operator $F$ given by \eqref{defF1} associated with the MFG in \eqref{mp}.  
Assume that $H$ is strictly convex,  $g$ is strictly increasing, and $\phi>0$ in $\Tt^d$. Then,
there is at most one solution $(m, u)\in C^2(\Tt^d; \Rr_0^+)\times C^1(\Tt^d)$ to \eqref{mp}. 
\end{pro}
\begin{proof}
Let $(m, u)\in C^2(\Tt^d; \Rr_0^+)\times C^1(\Tt^d)$ solve \eqref{mp}.
We first prove that $m>0$.
Assume by contradiction 
that $\min_{x\in \Tt^d} m=m(x^\ast)=0$ for some $x^\ast \in \Tt^d$.
Then, $\div(m(x^\ast)D_p H(D u(x^\ast)))=0$ since both $m$ and $Dm$ vanish at $x^\ast$.
Thus
\begin{equation*}
        0 = m(x^\ast)-\div(m(x^\ast)D_p H(D u(x^\ast)))-\phi(x^\ast)
=
        - \phi(x^\ast)<0,
\end{equation*}
which is a contradiction. Hence, $m$ is strictly positive in
$\Tt^d$. 

To prove uniqueness, let $(m_1, u_2)$ and $(m_2,u_2)$ be solutions to \eqref{mp} in 
$C^2(\Tt^d; \Rr_0^+)\times C^1(\Tt^d)$. Then $m_1$ and $m_2$ are strictly positive. Consequently, 
\eqref{A1mono}, combined with the convexity of $H$ and the monotonicity of $g$, 
gives
\begin{align*}
&H(D u_2)- H(Du_1)- D_p H(D u_1) \cdot (D u_2
        - D u_1)=0,\\
        &
        H(Du_1) - H(D u_2)-  D_p H(Du_2)\cdot (D u_1-D
        u_2)=0,\\
        &(m_1-m_2)(g(m_1)- g(m_2)) =0.
\end{align*}
Therefore, the strict convexity of $H$ gives $Du_1=Du_2$ and the strict monotonicity of $g$ 
gives $m_1=m_2$. Finally, the Hamilton--Jacobi equation combined with these facts gives $u_1=u_2$.
\end{proof}

\begin{remark}
The uniqueness for the second-order stationary MFG in \eqref{exs1} is handled similarly. For the second-order time-dependent case, the argument must be modified slightly: we require $m_0>0$
and use the parabolic maximum principle to get $m>0$.  
\end{remark}

\section{Minty's method and regularization of monotone operators}\label{exmono}

Before considering MFGs, we focus on monotone operators and explain 
Minty's method for studying the kernel of such
operators. The main idea of this method is to construct a convenient
monotone approximation for which the kernel is simpler to study. 
Then, using the  monotonicity
and asymptotic analysis, we build a weak solution for the original problem.
 We refer the reader to \cite{Eva} for a detailed
description of Minty's method, where this method
is used
for finding weak solutions for nonlinear PDEs.

Let $H$ be a Hilbert space and assume that $F\colon D(F)\subset
H \to H$ is a  monotone operator (see \eqref{def:mono}). The study
of the kernel of \(F\) is often formulated as a (strong) variational
inequality: find $u\in D(F)$ such that for all $v \in D(F)$,
we
have\begin{equation}\label{vari3}
(F(u), u- v) \leq 0.
\end{equation}
Note that if \(D(F)=H\), then taking \(v=u\pm w\) for an arbitrary
\(w\in H\), we conclude that \eqref{vari3} is equivalent to finding
\(u\in H\) such that \(F(u)=0\). 

In many instances, the (strong) variational
inequality
is either hard to address or does not have solutions, which motivates
the so-called  weak variational inequality associated with
\(F\). This weak variational inequality consists of finding 
$u\in H$ such that for all $v \in D(F)$,
we
have\begin{equation}\label{wvari3}
(F(v), u- v) \leq 0.
\end{equation}
A weak
solution to \eqref{vari3} is 
an element $u\in H$ 
 satisfying \eqref{wvari3}.
 Because of the monotonicity of  \(F\), any (strong)
solution to \eqref{vari3}  is also a weak solution to \eqref{vari3}.
Conversely, if  \(F\) is continuous and \(D(F)\) is convex, then
weak solutions belonging to \(D(F)\) are also strong solutions.
In fact, let \(u\in D(F)\) be a weak
solution to \eqref{vari3}, let   \(\lambda \in (0,1)\), fix \(\bar
v\in D(F)\), and set \(v= \lambda u + (1-\lambda)\bar v \in D(F)\).
Then, by \eqref{wvari3}, we have
\begin{equation*}
\begin{aligned}
0\geq \big(F(\lambda u + (1-\lambda)\bar v), (1-\lambda)(u-\bar
v)\big). \end{aligned}
\end{equation*}
Dividing the preceding inequality by
\(1-\lambda\), and then letting \(\lambda\to 1^-\), we obtain
\begin{equation*}
\begin{aligned}
0\geq(F(u), u- \bar v) 
\end{aligned}
\end{equation*}
by the continuity
of \(F.\)   Because  \(\bar
v\in D(F)\) is arbitrary, we conclude that \eqref{vari3} holds;
that
is, \(u\) is also a strong solution to \eqref{vari3}. 

A common method to finding weak solutions  to \eqref{vari3} is
Minty's
method, which is hinged on finding a convenient monotone approximation
of \(F\). More precisely,  assume that there exists  a dense subset, \(D\subset
D(F)\),   and, for each \(k\in\Nn\), there
exists an operator, $F_k\colon D\subset H\to H$, satisfying the  following
 conditions: 
\begin{enumerate}
\item \textit{(monotonicity)} 
for all $v$, $u \in D$, we have
\begin{equation}\label{mono3-2}
(F_k(u) - F_k(v), u- v) \geq 0;
\end{equation}
\item \textit{(convergence to \(F\)}) for all $v\in D$, we have
\begin{equation}\label{app3-1}
F_k(v)\to F(v)\ \ \mbox{strongly in }H;
\end{equation}
\item \textit{(lower semicontinuity-type condition)} for all \(v\in D(F)\), there exists
$\{v_k\}_{k=1}^\infty\subset D$ such that 
\begin{equation}\label{eq:conttype}
        \begin{aligned}
                v_k\to v \text{ in } H \text{ and } F(v_k) \rightharpoonup F(v) 
                \text{  weakly in } H.
        \end{aligned}
\end{equation}
\item \textit{(strong solutions)} there exists \(u_k\in D\) such
that for all $v\in D$, we have
\begin{equation}\label{ss3-1} 
(F_k(u_k), u_k- v)\leq0;
\end{equation}
\item \textit{(compactness)} given $\{u_k\}_{k=1}^\infty$  as in \eqref{ss3-1},
 there exists \(\bar u\in H\)
such
that, extracting a subsequence if necessary,  we have
\begin{equation}\label{wcuktou} 
 u_k \rightharpoonup \bar u \ \ \mbox{weakly in }H;
\end{equation}
\end{enumerate}

\begin{pro}[Minty's method]
Let $F$ be a monotone operator in a Hilbert space $H$ and $F_k$ be as above satisfying conditions 1--5. 
Then, any $\bar u\in H$ given by \eqref{wcuktou} is a weak solution to \eqref{vari3}.
\end{pro}
\begin{proof}
By combining \eqref{mono3-2} with \(u\) replaced by $u_k$
and \eqref{ss3-1}, we obtain that \((F_k(v), u_k- v) \leq 0\)
for all \(v\in D\).
Letting \(k\to\infty\) in this estimate, \eqref{app3-1} and \eqref{wcuktou}
yield
\begin{equation}\label{wsvi3-1}
(F(v), \bar u- v) \leq 0
\end{equation}
for all \(v\in D\). Finally, using \eqref{eq:conttype}, we conclude that
\eqref{wsvi3-1} holds for all \(v\in
D(F)\), which shows that \(\bar u\) is a weak solution
 to \eqref{vari3}.
\end{proof}

In the following subsections, we illustrate how to construct such
monotone
approximations, \(F_k\).

\subsection{Regularization strategies}\label{streg}
We now show how to construct monotone approximations as
mentioned above. We will use a regularization method based on
higher-order perturbations.  Rather than addressing MFG models, we consider
a simpler linear operator to illustrate the main ideas. In the next section, we examine operators that arise in MFGs. 
 
To simplify, assume that the domain is the $d$-dimensional torus,
\(\Tt^d\), and fix ${\bf b} \in C^1 (\Tt^d; \Rr^d)$ with $\div
({\bf
b}) =0$ and $f\in C^1(\Tt^d)$ with $f\not\equiv 0$.   Let \(F:H^1(\Tt^d)\subset
L^2(\Tt^d) \to L^2(\Tt^d)\) be defined
by
\begin{equation}
\label{eq:egF}
\begin{aligned}
F(u)=  u - {\bf b} \cdot Du - f
\end{aligned}
\end{equation}
for \(u\in H^1(\Tt^d)\). 
Consider the problem of finding \(u\in H^1(\Tt^d)\) such
that $F(u)=0$. Observe that $u\equiv 0$
is not a solution
to this problem because  $f\not\equiv 0$.
Moreover, we cannot apply the Lax--Milgram theorem
directly to solve \eqref{eq:egF} because the bilinear form associated with
\eqref{eq:egF} is neither continuous in $L^2(\Tt^d)$ and nor coercive in
$H^1(\Tt^d)$.
Instead, we use Minty's method as described above, for which we start by establishing the monotonicity of \(F\). 

\begin{pro}\label{prop:Ffmon}
The operator $F$ given by  \eqref{eq:egF} is monotone in $L^2(\Tt^d)$.
\end{pro}
\begin{proof}
 Given
$u$,  $v \in H^1(\Tt^d) $,  integrating by parts and Proposition~\ref{prop:Ffmon} yield
\begin{align*}
\left(F(u) -F(v),  u-v \right)
&=
\int_{\Tt^d} \big( (u-v)^2 - {\bf b} \cdot D(u-v) (u-v) \big)\,\dx
=
\int_{\Tt^d} \Big( (u-v)^2 - \frac{1}{2}{\bf b} \cdot D (u-v)^2
\Big)\,\dx\\
&=
\int_{\Tt^d} \Big( (u-v)^2 + \frac{1}{2}\div({\bf b}) (u-v)^2
\Big)\,\dx
=
\int_{\Tt^d} (u-v)^2 \,\dx \geq 0.\qedhere
\end{align*}
\end{proof}

To use Minty's method, we must construct an approximation
of \(F\) that preserves monotonicity. One possibility is to choose, for $\epsilon >0$, 
the regularized operator
\begin{equation}
\label{eq:Flin}
F_\epsilon(u)= F(u) + \epsilon (u +\Delta^2 u)
\end{equation}
with domain $D=H^4(\Tt^d)$.

\begin{pro}
The operator $F_\epsilon$ given by \eqref{eq:Flin} satifies conditions \eqref{mono3-2}, \eqref{app3-1}, 
\eqref{eq:conttype}, and \eqref{wcuktou}.       
\end{pro}
\begin{proof}

For all $u$, $v \in H^{4}(\Tt^d)$,
an integration by parts yields
\begin{equation*}
\left(F_\epsilon(u) - F_\epsilon (v), u-v\right)
=
\left(F(u) - F (v), u-v\right) + \epsilon\big(\| u - v \|_{L^2(\Tt^d)}^2+
\|\Delta u -\Delta v\|_{L^2(\Tt^d)}^2\big)
\geq 0.
\end{equation*}
Thus,  $F_\epsilon$ is (strictly) monotone over $L^2(\Tt^d)$. Hence, \eqref{mono3-2} holds. 

Because $\epsilon (v+\Delta^{2}v)\to 0$ strongly in $L^2$ for $v\in H^{4}(\Tt^d)$, \eqref{app3-1}
also holds. 

Next, by density of $H^{4}(\Tt^d)$ in $H^1(\Tt^d)$,
for any $v\in H^1(\Tt^d)$, 
 there exists a sequence $\{v_k\}_k$ in $H^{4}(\Tt^d)$
that converges strongly in $H^1(\Tt^d)$ to $v$.
Hence,  $v_k \to v$ in $L^2(\Tt^d)$ and $F(v_k) \rightharpoonup F(v)$
weakly in $L^2(\Tt^d)$. Thus, \eqref{eq:conttype} holds.

Assume that there exists $u_\epsi\in H^{4}(\Tt^d)$ such that
\begin{equation}\label{req4-1}
F_\epsilon(u_\epsilon) = 0 \enspace\text{ in } \Tt^d.
\end{equation}
Then, multiplying both sides of \eqref{req4-1}
by $u_\epsi$, integrating by parts, and using the condition
$\div({\bf b})=0$, we get
\begin{equation*}
0=
\int_{\Tt^d}\big( F(u_\epsi)u_\epsi + \epsilon\big( u_\epsi^2
+ (\Delta u_\epsi)^2\big)\big)\,\dx
=
\int_{\Tt^d}\big( u_\epsi^2 + f(x)u_\epsi + \epsilon\big( u_\epsi^2
+ (\Delta  u_\epsi)^2\big)\big)\,\dx.
\end{equation*}
Applying Young's inequality, we conclude that there exists a
positive constant, $C$, such that 
\begin{equation*}
\int_{\Tt^d}\big( u_\epsi^2 + \epsilon\big( u_\epsi^2 + (\Delta
 u_\epsi)^2\big)\big)\,\dx\leq C.
\end{equation*}
Because $C$ is independent of $\epsilon$, we obtain uniform a
priori estimates in $H^2(\Tt^d)$ for the solutions \(u_\epsi\), from which \eqref{wcuktou} holds by the Rellich--Kondrachov theorem. 
\end{proof}

\begin{pro}
\label{vari}
The operator $F_\epsilon$ given by \eqref{eq:Flin} satifies  condition \eqref{ss3-1}.
\end{pro}

In the subsequent sections, we describe three approaches to prove the preceding propositon
that establishes the existence of solutions to the regularized problem. These are 
a variational approach, a bilinear
form approach, and a continuity method approach.

\subsubsection{Variational problem approach}\label{vpa4-1}
Given $u\in H^{1}(\Tt^d)$,
we define a functional, $G_{u}\colon H^{2}(\Tt^d) \to \Rr$,
by
\begin{equation}\label{defGu}
G_{u}(v)=\int_{\Tt^d} \Big(\frac{\epsilon}{2} (v^2 + (\Delta
v )^2) + F(u) v\Big)\,\dx, \quad v\in H^{2}(\Tt^d),
\end{equation}
where \(F\) is given by \eqref{eq:egF}.

\begin{pro}\label{prop35}
Fix $u_0\in H^{1}(\Tt^d)$ and  set $G= G_{u_0}$.
The variational problem of finding \(\bar v
\in H^{2}(\Tt^d)\) satisfying 
\begin{equation}\label{var4-1} 
G(\bar v) = \inf_{v \in H^{2}(\Tt^d)} G(v)
\end{equation}
has a unique minimizer. 
\end{pro}
\begin{proof}
The functional $G$ is weakly lower-semicontinuous in $H^2(\Tt^d)$. To apply the direct 
method in the calculus of variations, it suffices to show that it is coercive. 
This follows from the inequality
\begin{align*}
        G(v)
        &=
        \int_{\Tt^d} \Big(\frac{\epsilon}{2} (v^2 + (\Delta
v )^2) + F(u_0) v\Big)\,\dx
        \geq
        \int_{\Tt^d} \Big(\frac{\epsilon}{2} (v^2 + (\Delta
v )^2) - C |v|\Big)\,\dx\\
        &\geq
        \int_{\Tt^d} \Big(\frac{\epsilon}{2} (v^2 + (\Delta
v )^2) -\frac{\epsilon}{4}v^2 - \frac{4C}{\epsilon}\Big)\,\dx
        \geq
        C_\epsilon \|v\|_{H^{2}(\Tt^d)}^2 - \frac{C}{\epsilon},
\end{align*}
which is a consequence of Young's and the Gagliardo--Nirenberg interpolation
inequalities. Because $G$ is strictly convex, the minimizer is unique. 
\end{proof}

The preceding proposition enable us to define an operator  $\Uu \colon
H^{1}(\Tt^d)\to H^{1}(\Tt^d)$ by
\begin{equation}
        \label{udef}
        \Uu[u_0] = \bar v,
\end{equation}
where 
 \(u_0\in H^{1}(\Tt^d)\) and
$\bar v$ is the solution of \eqref{var4-1} given by Proposition~\ref{prop35}. 
To prove Proposition \ref{vari}, 
we aim at showing that $\Uu$ has a fixed point. 
Note that the range of $\Uu$ is $H^2(\Tt^d)$, hence any fixed point of $\Uu$ is
automatically in $H^2(\Tt^d)$.
Once the existence of a fixed point is established, it satisfies \eqref{ss3-1}
by the following proposition. 

\begin{pro}
        \label{36}
Let $u\in H^2(\Tt^d)$ be a fixed point of $\Uu$. Then,  
\begin{equation}\label{vari4-1}
        \int_{\Tt^d}\Big(\epsilon\big(u (\omega-u) + \Delta
        u\Delta (\omega-u)\big) + F(u)(\omega-u)\Big)\,\dx\geq 0
\end{equation}
for all $\omega \in H^{2}(\Tt^d)$.
\end{pro}
\begin{proof}
Fix $\omega \in H^{2}(\Tt^d)$ and let $I\colon [0,1] \to \Rr$ be given by 
\begin{equation*}
I(t)= G(u +t (\omega-u)), \quad t\in[0,1].
\end{equation*}
Because $u$ is a minimizer, we have $I(0)\leq I(t)$. Consequently,
\begin{equation*}
0\leq \frac{1}{t}\big(I(t) - I(0)\big)
=
\int_{\Tt^d}\Big(\epsilon\big(u (\omega-u) + \Delta
u\Delta (\omega-u)\big) + F(u)(\omega-u)\Big)\,\dx
+ \Oo(t)
\end{equation*}
for all $t\in (0,1]$. Letting  $t\to 0$ in the preceding inequality,
 we conclude that \(u\) satisfies \eqref{vari4-1}.
\end{proof} 
\begin{remark}
\label{remel}
Arguing as in the previous proposition for $\bar v=\Uu(u_0)$, 
we obtain that 
\[
\int_{\Tt^d}\Big(\epsilon\big(\bar v (\omega-\bar v) + \Delta
\bar v\Delta (\omega-\bar v)\big) + F(u_0)(\omega-\bar v)\Big)\,\dx\geq 0. 
\]      
Further, by taking $t\in [-1,0]$ in the preceding proof gives
\[
\int_{\Tt^d}\Big(\epsilon\big(\bar v (\omega-\bar v) + \Delta
\bar v\Delta (\omega-\bar v)\big) + F(u_0)(\omega-\bar v)\Big)\,\dx= 0, 
\]      
which is the weak form of the Euler--Lagrange equation for the functional $G$.
In particular, any fixed point of $\Uu$ in $H^2(\Tt^d)$ is a weak solution
of
\[
\epsilon(u+\Delta^2 u)+F(u)=0.
\]
Accordingly,  $\Delta^2 u\in L^2(\Tt^d)$, which gives $u\in H^4(\Tt^d)$.
\end{remark}

To prove that \(\Uu\) has a fixed point, we use 
the following version of Schaefer's fixed-point theorem (see,
for instance,
 \cite[Theorem~6.2]{FGT1} for its proof).

\begin{theorem}\label{Sch4-1}
Let $\mathcal{X}$ be a closed convex subset of a Banach space
such that  $0\in \mathcal{X}$. Assume that $\Uu\colon \mathcal{X}
\to \mathcal{X}$ is continuous and compact. Also, assume that
the set
\begin{equation*}
\big\{
w\in \mathcal{X}\ |\ w=\lambda\, \Uu[w]\ \ \mbox{for some }\lambda
\in [0,1]
\big\}
\end{equation*}
is bounded. Then, there exists a fixed point, $w\in \mathcal{X}$,
such that $w=\Uu[w]$.
\end{theorem}
To apply the preceding theorem, we first  prove the continuity and compactness of $\Uu$. 
\begin{pro}
\label{39}
The operator  $\Uu$ in \eqref{udef} is 
continuous and compact in $H^{1}(\Tt^d)$.
\end{pro}
\begin{proof}
First, we establish continuity.
Let $\{u_n\}_{n=1}^\infty \subset  H^{1}(\Tt^d)$ and $u
\in H^{1}(\Tt^d)$  be such that $u_n\to u$ in \( H^{1}(\Tt^d)\).
Recalling \eqref{defGu}, set $G^\ast = G_{u}$ and
$G_n= G_{u_n}$.
Let
\begin{equation*}
\bar v = \Uu[u], \ \ \bar v_n =\Uu[u_n];
\end{equation*}
that is, $\bar v$ and $\bar v_n$ solve \eqref{var4-1}
with \(G\) replaced by $G^\ast$ and $G_n$, respectively.
Because $\bar v$ and $\bar v_n$ are minimizers over \(H^{2}(\Tt^d)\)
 of $G^\ast$ and $G_n$, respectively, we have
\begin{align*}
G^\ast (\bar v) + G_n(\bar v_n)
&\leq 
G^\ast\Big(\frac{\bar v + \bar v_n}{2}\Big) + G_n\Big(\frac{\bar
v + \bar v_n}{2}\Big)\\
&=
\frac{1}{2} G^\ast (\bar v) + \frac{1}{2} G^\ast (\bar v_n) 
+ \frac{1}{2} G_n(\bar v)  + \frac{1}{2} G_n(\bar v_n)
-
\int_{\Tt^d} \frac{\epsilon}{4}\Big[(\bar v - \bar v_n)^2 + (\Delta\bar
v - \Delta\bar v_n)^2\Big]\,\dx,
\end{align*}
which can be written as
\begin{align*}
\int_{\Tt^d} \frac{\epsilon}{4}\Big[(\bar v - \bar v_n)^2 + (\Delta\bar
v - \Delta\bar v_n)^2\Big]\,\dx
\leq 
\frac{1}{2} \big(
- G^\ast (\bar v) + G^\ast (\bar v_n) 
+ G_n(\bar v)  - G_n(\bar v_n)\big
).
\end{align*}
Using Young's inequality, we obtain that
\begin{align}
        \notag
\int_{\Tt^d} \frac{\epsilon}{4}\Big[(\bar v - \bar v_n)^2 + (\Delta\bar
v - \Delta\bar v_n)^2\Big]\,\dx
&\leq
\int_{\Tt^d}|\bar v - \bar v_n||F(u)-F(u_n)|\,\dx\notag\\
&\leq
\int_{\Tt^d}\frac{\epsilon}{8}(\bar v - \bar v_n)^2 \,\dx
+\frac{C}{\epsilon}
\|u - u_n\|^2_{H^1(\Tt^d)}.
\label{est}
\end{align}
 Hence, using the
Gagliardo--Nirenberg interpolation inequality and the convergence
 $u_n \to u$ in $H^{1}(\Tt^d)$, we conclude from the previous
estimate that $\bar v_n$  converges to $\bar v$ in $H^{2}(\Tt^d)$.
This proves that \(\Uu\) is continuous over \(H^{1}(\Tt^d)\).

Next, we establish compactness. For that, let $u_n$ be a bounded
sequence in $H^1(\Tt^d)$. We claim that the corresponding sequence $\bar v_n=\Uu(u_n)$ is pre-compact. 
By choosing $u=0$ and $\bar v=\Uu(0)$, inequality \eqref{est} combined with Gagliardo--Nirenberg estimate  gives that $\bar v_n$ is bounded
in $H^2(\Tt^d)$, hence pre-compact in $H^1(\Tt^d)$ by the Rellich--Kondrachov
theorem.
\end{proof}

\begin{pro}
\label{p310}
The operator  $\Uu$ in \eqref{udef} has a fixed point $\bar u\in H^1(\Tt^d)\cap  H^2(\Tt^d)$.
\end{pro}
\begin{proof}
To establish the existence of a fixed point, we apply Theorem \ref{Sch4-1}. For that,  
fix $\lambda \in [0,1]$ and assume that $u_\lambda$ satisfies
\begin{equation*}
u_\lambda=\lambda\, \Uu[u_\lambda].
\end{equation*}
If \(\lambda=0\), then \(u_\lambda=0\). 
If \(\lambda \in (0,1]\), then  arguing as in Remark \ref{remel}, 
we have 
\begin{equation*}
\int_{\Tt^d}\Big(\epsilon\big(u_\lambda (\lambda\omega-u_\lambda)
+ 
\Delta^k u_\lambda\Delta^k (\lambda \omega-u_\lambda)\big) 
+ 
\lambda F({u_\lambda})(\lambda \omega-u_\lambda)\Big)\,\dx=0
\enspace \text{
for all $\omega \in H^{2}(\Tt^d)$.}
\end{equation*}
Setting $\omega=0$,  integrating by parts with $\div({\bf b})=0$,
and using Young's inequality, we get
\begin{equation*}
\int_{\Tt^d}\Big( \lambda u_\lambda^2 + \epsilon\big( u_\lambda^2 + (\Delta
u_\lambda)^2\big)\Big)\,\dx\leq C,
\end{equation*}
where $C$ is independent of $\lambda$ and $\epsilon$.  Thus, using the
Gagliardo--Nirenberg interpolation inequality, 
$u_\lambda$ is bounded uniformly for $\lambda\in[0,1]$ in $H^2(\Tt^d)$ and hence in $H^1(\Tt^d)$. 
Therefore, by the compactness and continuity proven in Proposition~\ref{39}  combined with Theorem~\ref{Sch4-1}, we conclude that $\Uu$ has a fixed point
 $\bar u\in H^{1}(\Tt^d)\cap  H^2(\Tt^d)$. 
\end{proof}

\begin{proof}[Proof of Proposition \ref{vari}]
Proposition  \ref{vari} follows by combining the existence of a fixed point established in 
Proposition~\ref{p310}, its characterization in Proposition \ref{36}, and the $H^4$ regularity 
proven in Remark~\ref{remel}.
\end{proof}

%
\subsubsection{Bilinear form approach}\label{bfa4-2}

Now, we present an alternative proof of Proposition \ref{vari} using Lax--Milgram theorem 
applied to a suitable bilinear form. 

Fix $u_0 \in H^{1}(\Tt^d)$. 
Define a bilinear form $B\colon H^{2}(\Tt^d)\times H^{2}(\Tt^d) \to \Rr$ and a
linear functional
$F_{u_0}: H^{2}(\Tt^d)\to \Rr$ by
\begin{equation*}
B[u,v]=
\int_{\Tt^d}\epsilon\big(uv + \Delta u\cdot \Delta v\big)\,\dx
\enspace \text{ and } \enspace 
\langle F_{u_0}, v\rangle=
\int_{\Tt^d}F(u_0) v\,\dx
\end{equation*}
for \(u\), \(v\in H^{2}(\Tt^d)\). 

\begin{pro}
There exists 
$\bar u\in H^{2}(\Tt^d)$ such that  \begin{equation}\label{Bi4-1}
        B[\bar u,v]
        =
        -\langle F_{u_0}, v\rangle \enspace \text{ for all } 
        v\in H^{2}(\Tt^d).
\end{equation}
\end{pro}
\begin{proof}
By the Gagliardo--Nirenberg interpolation inequality, we can
find  a positive constant, $C_1$, such that the bilinear form
\(B\)
satisfies 
\begin{equation}
        \label{coerc}
B[u,u]\geq \epsilon C_1 \|u\|_{H^{2}(\Tt^d)}^2\enspace  \text{
for all } 
u\in H^{2}(\Tt^d).
\end{equation}
Moreover, using the Cauchy--Schwarz inequality, there exists
another positive constant, $C_2$, such that
\begin{equation}
        \label{bound}
|B[u,v]|\leq C_2\|u\|_{H^{2}(\Tt^d)}\|v\|_{H^{2}(\Tt^d)} \enspace
 \text{ for all } 
u, \,v\in H^{2}(\Tt^d).
\end{equation}
Because $u_0\in H^{1}(\Tt^d)$, $F(u_0)$ is a bounded in $L^2$. Thus,
$F_{u_0}$ is a bounded linear function in $H^2(\Tt^d)$. 
Therefore, by the Lax--Milgram theorem, there exists a unique
$\bar u\in H^{2}(\Tt^d)$ satisfying \eqref{Bi4-1}. 
\end{proof}

\begin{pro}
There exists \(u\in H^{4}(\Tt^d)\) solving
\begin{equation}\label{bfp}
        B[u,v]
        =
        -\langle F_u, v\rangle \enspace \text{ for all } 
        v\in H^{2}(\Tt^d).
\end{equation}
\end{pro}
\begin{proof}
To prove the proposition, we apply 
the fixed-point result in Theorem~\ref{Sch4-1}.
For that, 
let  $\tilde\Uu: H^{1}(\Tt^d)\to H^{1}(\Tt^d)$
be defined for each $u_0\in H^{1}(\Tt^d)$ by
\begin{equation*}
\tilde\Uu[u_0] = \bar u,
\end{equation*}
where $\bar u\in H^{2}(\Tt^d)$ is a weak solution to \eqref{Bi4-1}.

First, we prove that $\tilde\Uu$ is continuous over $H^{1}(\Tt^d)$.
Fix $\{u_n\}_{n=1}^\infty \subset H^{1}(\Tt^d)$ such that
$u_n \to u$ in $H^{1}(\Tt^d)$ for some \(u\in H^{1}(\Tt^d)\).
Set
\begin{equation*}
\bar u = \tilde\Uu[u]  \enspace \text{ and }\enspace
\bar u_n = \tilde\Uu[u_n].
\end{equation*}
We then have
\begin{align*}
B[\bar u, v]=-\left\langle  F_u, v\right\rangle  \enspace \text{ and }\enspace
B[\bar u_n, v]=-\left\langle  F_{u_n}, v\right\rangle  \enspace \text{ for all }
v\in H^{2}(\Tt^d).
\end{align*}
Subtracting these two equations with $v=\bar u - \bar u_n \in
H^{2}(\Tt^d)$,  Young's inequality yields
\begin{align}
\epsilon \big (\|\bar u - \bar u_n\|_{L^2(\Tt^d)}^2 + \|\Delta
(\bar u - \bar u_n)\|_{L^2(\Tt^d)}^2\big)
&\leq 
\int_{\Tt^d}|F_u-F_{u_n}||\bar u-\bar u_n|\,\dx\notag\\\label{eh2}
&\leq
\frac{\epsilon}{2}\|\bar u-\bar u_n\|_{L^2(\Tt^d)}^2
+
\frac{C}{\epsilon}\|u-u_n\|_{H^{1}(\Tt^d)}^2,
\end{align}
from which we obtain that $\|\bar u - \bar u_n\|_{L^2(\Tt^d)}^2
+ \|\Delta (\bar u - \bar u_n)\|_{L^2(\Tt^d)}^2 \to 0$ as $n\to
\infty$.
By the Gagliardo--Nirenberg interpolation inequality, it follows
that $\|\bar u - \bar u_n\|_{H^{2}(\Tt^d)}\to 0$ as $n\to\infty$.
Hence, $\tilde\Uu$ is continuous over $H^{1}(\Tt^d)$. 

To prove compactness, consider a bounded sequence 
 $\{u_n\}_{n=1}^\infty \subset H^{1}(\Tt^d)$ and let $\bar u_n=\tilde\Uu(u_n)$.
 Selecting $u=0$ and $\bar u=\tilde\Uu(0)$, the inequality \eqref{eh2} gives the 
 boundedness in $H^2(\Tt^d)$ of $\bar u_n$. Hence, 
by the Rellich--Kondrachov theorem, $\tilde\Uu$ is compact over
$H^{1}(\Tt^d)$.

Next,  assume that  $\lambda\in [0,1]$ and  let $u_\lambda\in H^{2}(\Omega)$ be
such that
\begin{equation*}
u_\lambda=\lambda\, \tilde\Uu[u_\lambda].
\end{equation*}
If \(\lambda=0\), then \(u_\lambda=0\). 
 If $\lambda\in (0,1]$, we use the
fact that $\frac{u_\lambda}{\lambda}$ is a solution to \eqref{Bi4-1}
with $F_{u_0}$ replaced by $F_{u_\lambda}$  to get
\begin{equation}\label{bi4-2}
\int_{\Tt^d}\epsilon(u_\lambda v + \Delta u_\lambda \Delta
v)\,\dx
=
-\int_{\Tt^d}\lambda (u_\lambda - b\cdot Du_\lambda - f(x))v\,\dx
\enspace \text{ for all }
v\in H^{2}(\Tt^d).
\end{equation}
Setting $v=u_\lambda$ in \eqref{bi4-2},  integrating by parts
and using the Gagliardo--Nirenberg interpolation inequality yields
\begin{equation*}
\|u_\lambda\|_{H^{2}(\Tt^d)}^2\leq C_\epsilon,
\end{equation*}
where $C_\epsilon>0$ is independent of $\lambda$. 

Applying Theorem~\ref{Sch4-1} and the regularity of solutions
to \eqref{Bi4-1} proven above,  we conclude that there exists
 $\bar u\in H^{2}(\Tt^d)$ such that \begin{align*}
\bar u=\tilde\Uu[\bar u].
\end{align*}
Arguing as in Remark \ref{remel}, we see that $\bar u\in H^4(\Tt^d)$.  
\end{proof}

\begin{proof}[Proof of Proposition \ref{vari}]
        Proposition  \ref{vari} follows directly from the preceding proposition as 
        \eqref{bfp} implies \eqref{ss3-1}. 
\end{proof}

\subsubsection{Continuity method}

Here, we discuss our last approach to prove Proposition \ref{vari}. We use a continuation argument 
and the implicit function theorem. 

For each $\lambda\in [0,1]$, we consider the following PDE:
  \begin{equation}\label{cm4-1}
\epsilon( u + \Delta^{2} u) + (u - {\bf b}\cdot D u - \lambda
f(x)) = 0, \enspace \text{ in } \Tt^d.
\end{equation}
Note that \eqref{req4-1} corresponds to \eqref{cm4-1} with \(\lambda=1\).

Define
\begin{equation}\label{defS}
S=\Big\{\lambda \in [0,1]\ |\ \mbox{there exists a solution in
$H^{4}(\Tt^d)$ to \eqref{cm4-1}}\Big\}.
\end{equation}
We prove next that $S$ is a non-empty set, relatively open and closed in $[0,1]$; hence, 
by connectedness, $S=[0,1]$. In particular, we  have a solution, $u$, to
\eqref{cm4-1} with $\lambda=1$. This approach to solving PDEs
is called the continuity method.

 When $\lambda=0$, $u\equiv 0$ is a solution to \eqref{cm4-1}.
Thus, 
$S$ is nonempty with $0 \in S$. Therefore, it remains to prove closeness and openness, which is done
in the following two propositions.

\begin{pro}
The set $S$ in \eqref{defS} is relatively closed in $[0,1]$. 
\end{pro}
\begin{proof}
Assume that $\{\lambda_n\}_{n=1}^\infty\subset
S$ is such that $\lambda_n \to \lambda$ for some $\lambda\in
[0,1]$. We want to show that \(\lambda\in S\). Let $\{u_n\}_{n=1}^\infty
\subset  H^4(\Tt^d)$ be the corresponding solutions of
\begin{align*}
\epsilon( u_n + \Delta^{2} u_n) + (u_n - {\bf b}\cdot D u_n
- \lambda_n f(x)) = 0, \enspace \text{ in } \Tt^d,
\end{align*}
for all \(n\in \Nn\). Multiplying both sides of the preceding
equation  by $u_n$ and integrating by parts, we get
\begin{equation*}
\epsilon(
\|u_n\|_{L^2(\Tt^d)}^2 + \|\Delta u_n\|_{L^2(\Tt^d)}^2
)
\leq
C,
\end{equation*}
where $C>0$ is independent of $n$. By the Gagliardo--Nirenberg
interpolation inequality, there exists a constant, $C>0$, such
that  $\|u_n\|_{H^{2}(\Tt^d)}\leq C$ for all $n\in \Nn$. Therefore,
extracting a subsequence if necessary, we have $u_n\rightharpoonup
u $ in $H^{2}(\Tt^d)$ for some \(u\in H^{2}(\Tt^d)\).   

Given $v\in H^{2}(\Tt^d)$, $u_n$ satisfies
\begin{equation}\label{eq4-3-1}
\int_{\Tt^d}\epsilon(u_n v + \Delta u_n \Delta v)\,\dx=\int_{\Tt^d}(-
u_n +{\bf b}\cdot Du_n - \lambda_n f(x))v\,\dx.
\end{equation}
Then, taking the limit as \(n\to\infty\) in \eqref{eq4-3-1},
the convergences  $u_n \rightharpoonup u$ in $H^{2}(\Tt^d)$
and $\lambda_n\to\lambda$ in \(\Rr\) imply that
\begin{align*}
\int_{\Tt^d}\epsilon(u v + \Delta u \Delta v)\,\dx=\int_{\Tt^d}(-
u +{\bf b}\cdot Du - \lambda f(x))v\,\dx.
\end{align*}
Because $v\in H^{2}(\Tt^d)$ is arbitrary, we conclude that \(u\)
is a weak solution to \eqref{cm4-1} in $H^2(\Tt^d)$. 
By elliptic regularity, \(u\in H^4(\Tt^d)\).
Thus, $\lambda \in S$, which proves  that $S$ is closed.
\end{proof}

\begin{pro}
The set $S$ in \eqref{defS} is relatively open in $[0,1]$. 
\end{pro}
\begin{proof}

 Fix $\lambda_0
\in S$, and   consider the Fr\'echet derivative, $\Ll: H^{4}(\Tt^d)\to
L^{2}(\Tt^d)$, of the map \(u\in H^{4}(\Tt^d) \mapsto\epsilon(
u + \Delta^{2} u) + (u - {\bf b}\cdot Du - \lambda_0
f)\in L^2(\Tt^d)   \). We have 
\begin{equation*}
\Ll(v)=\epsilon(v + \Delta^{2} v) + (v -{\bf b}\cdot Dv), \enspace
v\in H^{4}(\Tt^d).
  \end{equation*}
We aim at proving that  $\Ll$ is an isomorphism from $H^{4}(\Tt^d)$
to $L^{2}(\Tt^d)$, in which case  we conclude that $S$ is relatively
open by the implicit function theorem
in Banach spaces (see, for example, \cite{Die}).

To prove that $S$ is injective, assume that $v_1$, $v_2\in H^{4}(\Tt^d)$
satisfy $\Ll(v_1)=\Ll(v_2)$. Then, multiplying $\Ll(v_1)-\Ll(v_2)$
by $v_1-v_2$, integrating over \(\Tt^d\), an   integrating by
parts and  the  condition $\div({\bf b})=0$ imply that
\begin{equation*}
\|v_1-v_2\|_{L^2(\Tt^d)}^2 + \epsilon (\|v_1-v_2\|_{L^2(\Tt^d)}^2
+\|\Delta(v_1-v_2)\|_{L^2(\Tt^d)}^2)=0,
\end{equation*}
from which we get $v_1-v_2\equiv 0$. Hence, $\Ll$ is injective.

Next, we prove that $\Ll$ is surjective.
 Fix $w_0\in L^{2}(\Tt^d)$. Using the Lax--Milgram theorem, we
prove next that there exists $v_0\in H^{4}(\Tt^d)$ such that
$\Ll (v_0)=w_0$.
Define $B:H^{2}(\Tt^d)\times H^{2}(\Tt^d)\to \Rr$ by
\begin{equation*}
B[v_1, v_2]=
\int_{\Tt^d}\big[\epsilon (v_1v_2+ \Delta v_1\Delta v_2)+v_1v_2
- {\bf b}\cdot Dv_1 v_2\big]\,\dx, \quad v_1, \, v_2\in H^{2}(\Tt^d).
\enspace 
\end{equation*}
Arguing as before, \(B\) is a bilinear form satisfying
\begin{equation*}
B[v,v]=\epsilon(\|v\|_{L^2(\Tt^d)}^2+ \|\Delta v\|_{L^2(\Tt^d)}^2)+\|v\|_{L^2(\Tt^d)}^2
\geq C_\epsilon \|v\|_{H^{2}(\Tt^d)}
\end{equation*}
for all \(v\in H^{2}(\Tt^d) \), where $C_\epsilon >0$ is independent
of $v$. Moreover, \begin{equation*}
|B[v_1,v_2]|\leq C\|v_1\|_{H^2(\Tt^d)}\|v_2\|_{H^2(\Tt^d)}
\end{equation*}
for all \(v_1\), \( v_2\in H^{2}(\Tt^d)\), where \(C>0\) is
independent
of \(v_1\) and \(v_2\). Therefore, by the Lax--Milgram theorem,
there exists 
$v_0\in H^{2}(\Tt^d)$ such that\begin{equation*}
B[v_0, w]=\int_{\Tt^d}w_0 \,w\,\dx \enspace \text{ for all }
w\in H^{2}(\Tt^d).
\end{equation*}  
Because $v_0\in H^{2}(\Tt^d)$,  we have $g=- v_0
+ {\bf b}\cdot Dv_0 + w_0  \in L^2(\Tt)$. Then, for all $w\in
H^{2}(\Tt^d)$, $v_0$ satisfies
\begin{equation*}
\int_{\Tt^d}\epsilon \big(v_0 w+ \Delta v_0 \Delta w\big)\,\dx
=
\int_{\Tt^d}g\, w\,\dx.
\end{equation*}
By the elliptic regularity theory, we have $v_0\in H^{4}(\Tt^d)$,
from which we conclude that $\Ll$ is surjective.
\end{proof}

Since $S$ is nonempty, relatively open, and closed in \([0,1]\),
we have $S=[0,1]$. Thus, there exists a solution, $u\in H^{4}(\Tt^d)$,
to \eqref{cm4-1} with $\lambda =1$.

\begin{proof}[Proof of Proposition \ref{vari}]
Proposition  \ref{vari} follows directly from the preceding results as 
        \eqref{cm4-1} implies \eqref{ss3-1}. 
\end{proof}

\section{Existence of solutions to stationary MFG problems}\label{rpMFG}
Here, we use Minty's method to construct weak solutions to stationary MFGs. The main difficulty in applying Minty's method is to build regularizations that preserve monotonicity, 
have uniform a priori bounds, and for which we can prove the existence of solutions. We show how to use the techniques developed in the prior section to achieve these goals for MFGs.

Our model problem is the MFG problem in \eqref{mp} with the corresponding operator
$F$ given by \eqref{defF1}. 
As usual in the MFG theory, we work under the following general assumptions. The first one
concerns the (mild) regularity of the problem data.

\begin{hyp}
        \label{a0}
The Hamiltonian, $H$, is of class $C^2$ and convex. The coupling, 
$g$, is  a non-negative continuous increasing function. The source term,
$\phi$, belongs to \( C(\Tt^d)$ and is such that $\phi> 0$ and $\int_{\Tt^d}\phi\,\dx=1$.
\end{hyp}
 
 The previous assumption is satisfied in many MFGs that appear in the literature.
 For example, $H(p)=\tfrac{|p|^2}{2}$,  $g(m)=m$, and  $\phi\equiv 1$ satisfy our requirements.

 The convexity of the Hamiltonian
reflects the fact that 
 the mean-field game is associated with a control problem.
 The monotonicity of $g$ encodes that 
 the agents are crowd-averse. 
The non-negativity assumption on $g$ is used to simplify the presentation. 
However, our techniques  may be adapted to handle, for example,  $g(m)=\log m$.  
 The source $\phi$ represents the incoming flow of agents, which we chose to be normalized to $1$. 

\begin{hyp}\label{grH5-1}
        There exist constants, $\alpha>1$ and
        $C>0$, such that
        \begin{equation*}
                - H(p) + D_p H(p)\cdot p\geq \frac{1}{C}|p|^\alpha - C.
        \end{equation*}
\end{hyp}
By the definition of the Hamiltonian through the Legendre transform in 
\eqref{hamlt}, the preceding assumption is a lower bound in the Lagrangian when written in terms of the momentum, $p$, variable. 
\begin{hyp}\label{grH5-2} Let $\alpha$ be as in Assumption \ref{grH5-1}. There exists a constant, $C>0$, such that 
        \begin{equation*}
                H(p)\geq \frac{1}{C}|p|^\alpha - C.
        \end{equation*}
\end{hyp}
In classical mechanics, the Hamiltonian is the energy of the system. The preceding bound imposes the growth of the energy as the momentum increases.
The last two assumptions are satisfied for $L(v)=\frac{|v|^{\alpha'}}{\alpha'}$ with corresponding Hamiltonian $H(p)=\frac{|p|^{\alpha}}{\alpha}$, where $\frac 1 \alpha+\frac 1 {\alpha'}=1$. 

\begin{hyp}\label{grg5-1}
        For each $\delta \in (0,1)$, there exists a constant, $C_\delta>0$, such
        that, for all $m\in L^1(\Tt^d)$ with $m\geq0$, we have
        \begin{equation*}
                \int_{\Tt^d} m  \,\dx + \int_{\Tt^d}g(m)\,\dx
                \leq
                \delta\int_{\Tt^d} mg(m) \,\dx + C_\delta.
        \end{equation*}
\end{hyp}
\begin{remark}
        In Assumption~\ref{grg5-1}, we assume that $g$ is non-negative. However,
        we may relax this  condition on $g$ and assume instead that \(g\)  is  bounded
        from below. An example of a function bounded from below satisfying the estimate
        in  Assumption~\ref{grg5-1} is given by $g(m)=m\ln m$ for \(m\in\Rr_+\).
        Moreover, though not bounded from below, the function \(g\) defined by \(g(m)=\log(m)\)
        satisfies Assumption~\ref{grg5-1}, and also Assumption~\ref{wcml1} below,
        provided that  \(m>0\) almost everywhere in \(\Tt^d\).\end{remark}

\begin{hyp}\label{wcml1}
        Assume that $\{m_j\}_{j=1}^\infty\subset L^1(\Omega)$ satisfies $m_j\geq
        0$ and 
        \begin{align*}
                \sup_{j\in \Nn}\int_{\Tt^d}m_j g(m_j)\,\dx <\infty.
        \end{align*}
        Then, there is a subsequence, $\{m_{j_i}\}_{i=1}^\infty$, that weakly converges
        in \(L^1(\Tt^d)\)  to some $m\in L^1(\Tt^d)$.
\end{hyp}

Observe that $F$ in  \eqref{defF1} is well defined in 
$D(F)=H^{2k}(\Tt^d; \Rr_0^+)\times H^{2k}(\Tt^d)$ taking values in $L^2(\Tt^d)\times L^2(\Tt^d)$ for  $k\in\Nn$ large enough. 
For our purposes, we work with 
$2k > \frac{d}{2} + 3$. 
Moreover, we see that $F$ is monotone in $L^{2}(\Tt^d)\times
L^{2}(\Tt^d)$ and  \eqref{A1mono} holds.  

The main goal of this section is to prove existence of  weak solutions
to \eqref{mp} in the sense of the following definition. 

\begin{definition}
A pair $(m, u)\in L^1(\Tt^d; \Rr_0^+)\times L^1(\Tt^d)$ is a weak solution 
 to \eqref{mp}, that is to 
\begin{equation}\label{mfg5-1}
F
  \begin{bmatrix}
      m \\
      u
  \end{bmatrix}
=0,
\end{equation}
if
        \begin{equation}
                \label{ws}
        \left\langle  
        F
        \begin{bmatrix}
                w \\
                v
        \end{bmatrix}
        ,
        \begin{bmatrix}
                m \\
                u
        \end{bmatrix}
        -  \begin{bmatrix}
                w \\
                v
        \end{bmatrix}
        \right\rangle 
        \leq
        0
\end{equation}
for  all $(w,v)\in H^{2k}(\Tt^d; \Rr_0^+)\times H^{2k}(\Tt^d)$.
\end{definition}

\begin{remark}
In this paper, we use $\langle\cdot, \cdot\rangle$ to denote the $L^2$ inner product. By abuse of notation, whenever $f$ and $g$ are functions, not necessarily in $L^2$ but whose product is in $L^1$, we write $\langle f, g\rangle=\int_{\Tt^d} fg\,\dx$. In the inequality \eqref{ws}, we use this convention. In particular, 
because $F(w,v)$ is continuous whenever
$(w,v)\in H^{2k}(\Tt^d; \Rr_0^+)\times H^{2k}(\Tt^d)$,  the corresponding integral in \eqref{ws} is well defined. 
\end{remark}

In the following theorem, we prove the existence of weak 
solutions. 

\begin{teo}\label{mt1}
        Under Assumptions~\ref{a0}--\ref{wcml1}, there exists a weak solution $(m,u)\in L^1(\Tt^d; \Rr_0^+)\times W^{1,\alpha}(\Tt^d)$ to \eqref{mp}.
\end{teo}

We discuss two approaches to prove this theorem. Both introduce a regularized MFG operator $F_\epsilon$
and establish the existence of a solution of 
\begin{equation}\label{mfg5-epsi}
        F_\epsi
        \begin{bmatrix}
                m_\epsi \\
                u_\epsi
        \end{bmatrix}
        =0,
\end{equation}
with $m_\epsilon\geq 0$.  In both cases, the fundamental difficulty  lies in proving the non-negativity
of  the population density, $m_\epsi$.
The first approach gives the existence of a weak solution of \eqref{mfg5-epsi} with $m_\epsilon\geq 0$
using a variational approach combined with a
bilinear form techniques. 
The proof corresponding to this approach requires the following additional assumption.
\begin{hyp}
\label{asup}
Let $\alpha$ be as in Assumption \ref{grH5-1}. If $\alpha \geq d$, let
$\alpha^*=+\infty$. Otherwise,
let $\alpha^*$ be the Sobolev conjugate exponent of $\alpha$,
\[
\frac{1}{\alpha^*}=\frac{1}{\alpha}-\frac{1}{d}.
\]
There exist $q<\alpha^*$ and  $C>0$ such that, for all \(m\in
L^1(\Tt^d)\),
\[
\int_{\Tt^d} m^{q'}\,\dx\leq C\int_{\Tt^d} m g(m)\,\dx+C,
\]where \(q'\) is the conjugate exponent of \(q\).
Moreover,
\begin{equation}
\label{eq:Hfromabove}
\begin{aligned}
H(p)\leq C(|p|^\alpha +1).
\end{aligned}
\end{equation}
\end{hyp}

The second approach constructs a smooth solution to  \eqref{mfg5-epsi} with $m_\epsilon>0$
using the continuity method and does not require the preceding assumption. Hence, Theorem \ref{mt1} is stated
under Assumptions~\ref{a0}--\ref{wcml1}.

Our first choice of $F_\epsilon$ is as follows. 
For  $\epsilon \in (0,1)$,
we set
\begin{equation}\label{rp5-1}
F_\epsi\begin{bmatrix}
      m \\
      u
  \end{bmatrix}
  = \begin{bmatrix}
      - u-H(Du) + g(m) - V + \epsilon(m+\Delta^{2k}m ) \\
      m-\div\big(m D_p H(Du)\big) -\phi + \epsilon(u+ \Delta^{2k} u)
  \end{bmatrix}
\end{equation}
for $(m, u)\in H^{4k}(\Tt^d; \Rr_0^+)\times H^{4k}(\Tt^d)$.

A key feature of this regularization is that it preserves  monotonicity.
In fact,
for all $(m_1,u_1)$, $(m_2,
u_2) \in H^{4k}(\Tt^d;\Rr_0^+)\times H^{4k}(\Tt^d)$, 
we have
\begin{equation}\label{eq:Fepsimon}
\begin{aligned}
&\left\langle 
F_\epsi
  \begin{bmatrix}
      m_1 \\
      u_1
  \end{bmatrix}
-
F_\epsi
  \begin{bmatrix}
      m_2 \\
      u_2
  \end{bmatrix}
,
  \begin{bmatrix}
      m_1 \\
      u_1
  \end{bmatrix}
-  \begin{bmatrix}
      m_2 \\
      u_2
  \end{bmatrix}
\right\rangle \\
&\quad=\left\langle 
F
  \begin{bmatrix}
      m_1 \\
      u_1
  \end{bmatrix}
-
F
  \begin{bmatrix}
      m_2 \\
      u_2
  \end{bmatrix}
,
  \begin{bmatrix}
      m_1 \\
      u_1
  \end{bmatrix}
-  \begin{bmatrix}
      m_2 \\
      u_2
  \end{bmatrix}
\right\rangle  +\epsi\int_{\Tt^d} \big[(m_1-m_2)^2 + (\Delta^km_1-\Delta^km_2)^2\big]\,\dx\geq
0, 
\end{aligned}
\end{equation}
where we used the monotonicity of \(F\).

\begin{remark}\label{rmk:notation}
In general,  if \((m,
u) \in H^{2k}(\Tt^d;\Rr_0^+)\times H^{2k}(\Tt^d)\), then \(F_\epsi[m,u]\) may not
belong to \(L^2(\Tt^d) \times L^2(\Tt^d)\) due to the presence of
the \(4kth\)-derivatives. However, to simplify the notation, given  $(m_1,u_1)$,
$(m_2,
u_2) \in H^{2k}(\Tt^d;\Rr_0^+)\times H^{2k}(\Tt^d)$,
we write
\begin{equation*}
\begin{aligned}
\left\langle 
F_\epsi
  \begin{bmatrix}
      m_1 \\
      u_1
  \end{bmatrix}
,
    \begin{bmatrix}
      m_2 \\
      u_2
  \end{bmatrix}
\right\rangle  \quad \text{in  place of}\quad 
\left\langle 
F
  \begin{bmatrix}
      m_1 \\
      u_1
  \end{bmatrix}
,
    \begin{bmatrix}
      m_2 \\
      u_2
  \end{bmatrix}
\right\rangle  +\epsi\int_{\Tt^d} \big(m_1m_2+ \Delta^km_1\Delta^km_2\big)\,\dx.
\end{aligned}
\end{equation*}
Note that   if \((m_1,
u_1) \in H^{4k}(\Tt^d;\Rr_0^+)\times H^{4k}(\Tt^d)\), then an integration by parts
shows that
these  two expressions coincide.

We further observe that, with this convention in mind, \eqref{eq:Fepsimon}
holds for all   $(m_1,u_1)$,
$(m_2,
u_2) \in H^{2k}(\Tt^d;\Rr_0^+)\times H^{2k}(\Tt^d)$.
\end{remark}

In our second regularization, we address the non-negativity constraint of $m$ 
by adding a penalty term, $p_\epsilon$, to the first equation. This technique was first used in this context in \cite{FG2}.
More precisely, for each \(\epsi\in(0,1)\), we fix a non-decreasing $C^\infty$ function, $p_\epsilon \colon (0,\infty)\to (-\infty, 0]$, such that
\begin{equation}\label{eq:defpen}
        {\it p}_\epsilon(t)=
        \begin{cases}
                -\frac{1}{t^{d+1}}  &\text{if }0<t\leq \frac{\epsilon}{2}\\
                0  &\text{if } t\geq \epsilon.
        \end{cases}
\end{equation}
Note that \(p_\epsilon(t)\to -\infty \) as   $t \to 0^+$.
Then, we consider the following alternative choice for $F_\epsilon$:
\begin{equation}\label{crp5-1x}
        F_\epsilon
        \begin{bmatrix}
                m \\
                u
        \end{bmatrix} =
        \begin{bmatrix}- u-H(Du) +g(m) - V + p_\epsilon(m)+ \epsilon(m+\Delta^{2k}m )\\
                m-\div\big(m D_p H(Du)\big) - \phi  + \epsilon(u+ \Delta^{2k} u) 
        \end{bmatrix}, 
\end{equation}
defined for $u$ and $m$ smooth with $m>0$. 
\begin{remark}
\label{remfemon}
As with the previous regularization in \eqref{rp5-1}, 
the regularization in 
\eqref{crp5-1x} is also monotone.  This can be seen following the same steps used to establish \eqref{eq:Fepsimon}.
\end{remark}

\subsection{Variational problem and bilinear form approaches}\label{vpbfa5-1}
Here, we combine the variational approach with the   bilinear form arguments
discussed in Subsections~\ref{vpa4-1}--\ref{bfa4-2}.

First, we choose $k \in \Nn$ such that $2k > \frac{d}{2} + 3 $,  fix $(m_1,
u_1)\in H^{2k-1}(\Tt^d;\Rr_0^+)\times H^{2k-1}(\Tt^d)$, and     define
\begin{equation*}
\begin{aligned}
&f_1[m_1, u_1]=-u_1 - H(D u_1) + g(m_1) - V, \\  
&f_2[m_1,u_1]= - m_1 + \div(m_1 D_p H(D u_1)) + \phi.
\end{aligned}
\end{equation*}
Consider the variational problem of finding
 $m \in H^{2k}(\Tt^d; \Rr_0^+)$ such that
\begin{equation}\label{vp5-1}
J_\epsi(m)
=
\inf_{w\in H^{2k}(\Tt^d; \Rr_0^+)} J_\epsi(w),
\quad \text{where }
J_\epsi(w)= \int_{\Tt^d}\Big(\frac{\epsilon}{2}\big(w^2 + (\Delta^k
w)^2\big) + f_1[m_1,u_1]w \Big)\,\dx.
\end{equation} 
Next,  we define
the bilinear form
\begin{equation*}
B_\epsi[v_1,v_2]
=
\epsilon\int_{\Tt^d}(v_1 v_2 +\Delta^k v_1 \Delta^k v_2)\,\dx \quad \text{for
$v_1$, $v_2 \in H^{2k}(\Tt^d)$,  }
\end{equation*}
and consider the problem of finding  $u\in H^{2k}(\Tt^d)$ such that
\begin{equation}\label{bf5-1}
B_\epsi[u, v]=\left\langle f_2[m_1, u_1] , v\right\rangle \ \ \mbox{for all } v\in H^{2k}(\Tt^d),
 \end{equation}
where, as before, $\left\langle \cdot, \cdot\right\rangle $ stands for the $L^2$-inner
product. This bilinear form satisfies the conditions in Lax--Milgram's theorem for the Hilbert space $H^{2k}(
\Tt^d)$. Namely, it is coercive and bounded as in \eqref{coerc} and \eqref{bound}.

To prove existence and uniqueness of solutions to these two problems, we
first observe that  $(m_1, u_1)\in C^1(\Tt^d)\times C^1(\Tt^d)$ by Morrey's
embedding
theorem. Hence, $f_1[m_1,u_1 ]\in C(\Tt^d)$ and $ f_2[m_1,u_1] \in L^2(\Tt^d)$, provided that Assumption \ref{a0} holds.
Therefore, we can argue as  in Subsections~\ref{vpa4-1} and \ref{bfa4-2}
to show the following proposition.
\begin{pro}
Suppose that Assumption \ref{a0} holds and that $(m_1,
u_1)\in H^{2k-1}(\Tt^d)\times H^{2k-1}(\Tt^d)$.
Then , there exists a unique solution to problems   \eqref{vp5-1}
and  \eqref{bf5-1}.
\end{pro}
 
Next, to address the regularized problem \eqref{mfg5-epsi} with $F_\epsilon$ given by \eqref{rp5-1}, we  define a map $A:H^{2k-1}(\Tt^d;\Rr_0^+)\times H^{2k-1}(\Tt^d)\to H^{2k-1}(\Tt^d;\Rr_0^+)\times H^{2k-1}(\Tt^d)$
by
setting, for each $(m_1,
u_1)\in H^{2k-1}(\Tt^d;\Rr_0^+)\times H^{2k-1}(\Tt^d)$,
\begin{equation*}
A
\begin{bmatrix}
      m_1 \\
      u_1
\end{bmatrix} = \begin{bmatrix}
      m_2 \\
      u_2
\end{bmatrix},
\end{equation*}
where $m_2\in H^{2k}(\Tt^d; \Rr_0^+)$ is the unique solution to  \eqref{vp5-1} and  $u_2\in
H^{2k}(\Tt^d)$ is the unique solution to \eqref{bf5-1}. As in Subsection~\ref{vpa4-1},
we want to show that \(A\) has a fixed point by using   
 Theorem~\ref{Sch4-1}. Thus, we need to establish the following result. 
 
 \begin{pro}
        \label{contcomp}
Suppose that Assumption \ref{a0} holds.
Then, the map $A$ is continuous and compact
over $H^{2k-1}(\Tt^d;\Rr_0^+) \times H^{2k-1}(\Tt^d)$. 
\end{pro}
\begin{proof}
To prove continuity, we take a convergent sequence $(m_n,u_n)$ in $H^{2k-1}(\Tt^d;\Rr_0^+)\times H^{2k-1}(\Tt^d)$. The maps $f_1(m_n,u_n)$ and $f_2(m_n,u_n)$ converge in $C(\Tt^d)$ and in $L^2(\Tt^d)$, 
respectively. Thus, 
it suffices to argue as in  Subsections \ref{vpa4-1} and \ref{bfa4-2}
to establish continuity. 

To address compactness, we show next that the solutions to  \eqref{vp5-1}
and   \eqref{bf5-1} are automatically in $H^{2k}$ and are bounded in terms of the $L^2$ norms of $f_1$ and $f_2$, respectively. In turn, these norms of $f_1$ and $f_2$ 
are bounded by the norms of $m_1$ and $u_1$ in $H^{2k-1}$. For the solution $m_2$ of the variational problem \eqref{vp5-1},  
we have
\[
 \int_{\Tt^d}\Big(\frac{\epsilon}{2}\big(m_2^2 + (\Delta^k
m_2)^2\big) + f_1[m_1,u_1]m_2 \Big)\,\dx\leq 0
\]
by using $w=0$ as a competitor in \eqref{vp5-1}. By the Cauchy inequality, 
\[
 \int_{\Tt^d}f_1[m_1,u_1]m_2\,\dx \leq C_\epsilon \|f_1\|_{L^2(\Tt^d)}^2+\frac \epsilon 4 \|m_2\|^2_{L^2(\Tt^d)}.
\]
Combining the two previous inequalities with the Gagliardo--Nirenberg inequality, we obtain $\|m_2\|^2_{H^{2k}(\Tt^d)}\leq C_\epsilon \|f_1\|_{L^2(\Tt^d)}^2$.
For the solution $u_2$ of \eqref{bf5-1}, we use $v=u_2$ in \eqref{bf5-1} and argue analogously to get
 $\|u_2\|^2_{H^{2k}(\Tt^d)}\leq C_\epsilon \|f_2\|_{L^2(\Tt^d)}^2$.
By the Rellich--Kondrachov theorem, we have $ H^{2k}(\Tt^d; \Rr_0^+)\times H^{2k}(\Tt^d) \subset\subset
 H^{2k-1}(\Tt^d;\Rr_0^+) \times H^{2k-1}(\Tt^d)$.  This compact embedding combined  with the preceding a priori estimates in \(H^{2k}(\Tt^d) \times H^{2k}(\Tt^d)\) on the solutions 
 allow us to conclude that  $A$ is compact
over  $H^{2k-1}(\Tt^d;\Rr_0^+) \times H^{2k-1}(\Tt^d)$.
\end{proof}

Next, we address the a priori bound needed to apply Schaeffer's fixed point theorem. 

\begin{pro}
        \label{lambdasolprop}
Suppose that Assumptions \ref{a0}--\ref{grg5-1} hold.
Fix $\lambda \in [0,1]$ and assume that $(m_\lambda,u_\lambda)$
satisfies
\begin{equation}
        \label{lambdasol}
\begin{bmatrix}
      m_\lambda \\
      u_\lambda
\end{bmatrix}
=
\lambda A
\begin{bmatrix}
      m_\lambda \\
      u_\lambda
\end{bmatrix}.
\end{equation}
Then, $(m_\lambda, u_\lambda)$ is uniformly bounded in $H^{2k}(
\Tt^d)\times H^{2k}(\Tt^d)$
with respect
to $\lambda$. 
\end{pro}
\begin{proof}
If $\lambda=0$, then $(m_\lambda,u_\lambda)\equiv 0$. If $\lambda \in (0,1]$,
then $\frac{m_\lambda}{\lambda}$ solves \eqref{vp5-1} with $f_1[m_1, u_1]$
replaced by $f_1[m_\lambda, u_\lambda]$. Consequently,
for all $w\in H^{2k}(\Tt^d; \Rr_0^+)$,
we have\begin{equation}\label{vi5-1}
\int_{\Tt^d} \Big(\epsilon \big(m_\lambda(w - m_\lambda)+ \Delta^k m_\lambda
\Delta^{k}(w- m_\lambda) \big) + \lambda f_1[m_\lambda, u_\lambda](w-m_\lambda)\Big)\,\dx
\geq 0.
\end{equation}

Similarly, $\frac{u_\lambda}{\lambda}$ solves \eqref{bf5-1} with $f_2[m_1, u_1]$
replaced by $f_2[m_\lambda, u_\lambda]$; that is, 
\begin{equation}\label{bf5-1l}
        B_\epsi[u_\lambda, v]=\lambda\left\langle f_2[m_\lambda, u_\lambda] , v\right\rangle \ \ \mbox{for all } v\in H^{2k}(\Tt^d).
\end{equation}

Taking $w=0$ in \eqref{vi5-1} and  $v=u_\lambda$ in \eqref{bf5-1l} 
and combining the resulting estimates with the convexity of $H$, Assumptions~\ref{grH5-1},
\ref{grH5-2}, and \ref{grg5-1}, and the monotonicity of the system, we obtain
the bound
\begin{equation}\label{apriori5-1}
\lambda \int_{\Tt^d} \big(m_\lambda g(m_\lambda) + m_\lambda|Du_\lambda|^\alpha+\phi |Du_\lambda|^\alpha\big)\,\dx
+
\epsilon\int_{\Tt^d} \big(m_\lambda^2 + (\Delta^k m_\lambda)^2 + u_\lambda^2
+ (\Delta^k u_\lambda)^2\big)\,\dx
\leq C,
\end{equation}
where $C>0$ is independent of $\lambda$ and $\epsilon$. Because $\phi$ and $g$ are non-negative by Assumption \ref{a0},
the left-hand side of \eqref{apriori5-1} is non-negative. Therefore,  the Gagliardo--Nirenberg
inequality yields uniform $H^{2k}$ bounds for $(m_\lambda, u_\lambda)$ with
respect to $\lambda$.
\end{proof}

\begin{corollary}
        \label{cor49}
Suppose that Assumptions \ref{a0}--\ref{grg5-1} hold.
Then, there exists $(m_\epsi,u_\epsi)\in H^{2k}(\Tt^d;\Rr^+_0)\times H^{2k}(\Tt^d)$
satisfying 
\begin{equation}\label{eq:wsepsi}
        \begin{bmatrix}
                m_\epsi \\
                u_\epsi
        \end{bmatrix}
        =
        A
        \begin{bmatrix}
                m_\epsi \\
                u_\epsi
        \end{bmatrix}
\end{equation}
and the bound
\begin{equation}
\label{thebound}
\int_{\Tt^d} \big(m_\epsi g(m_\epsi) + m_\epsi|Du_\epsi|^\alpha+\phi |Du_\lambda|^\alpha\big)\,\dx
+
\epsilon\int_{\Tt^d} \big(m_\epsi^2 + (\Delta^k m_\epsi)^2 + u_\epsi^2
+ (\Delta^k u_\epsi)^2\big)\,\dx
\leq C, 
\end{equation}
where $C$ is independent of $\epsilon$.
In particular, $(m_\epsilon, u_\epsilon)$ is a weak solution to \eqref{mfg5-epsi}. \end{corollary}

\begin{proof}
By Proposition \ref{contcomp}, the map $A$ is continuous and compact
over $H^{2k-1}(\Tt^d;\Rr_0^+) \times H^{2k-1}(\Tt^d)$.  Further, by Proposition  \ref{lambdasolprop}, we have uniform bounds in $H^{2k-1}(\Tt^d;\Rr_0^+) \times H^{2k-1}(\Tt^d)$ for solutions
to \eqref{lambdasol}.   
 Consequently, for each \(\epsi\in(0,1)\), we can apply
Theorem~\ref{Sch4-1} to obtain a fixed point, $(m_\epsi,u_\epsi)$, with \(m_\epsi\geq 0\), satisfying \eqref{eq:wsepsi}.
Moreover, 
by the preceding proposition, $(m_\epsi,u_\epsi)\in H^{2k}(\Tt^d)\times H^{2k}(\Tt^d)$ and using \eqref{apriori5-1} with $\lambda=1$ gives \eqref{thebound}.

From the above, it follows that $(m_\epsilon, u_\epsilon)$ is a solution to the variational inequality associated with \eqref{mfg5-epsi}; that is, for all 
$(w,v)\in H^{2k}(\Tt^d;\Rr_0^+)\times
H^{2k}(\Tt^d)$, we have
\begin{align}\label{solmp5-1}
        \left\langle 
        F_\epsilon
        \begin{bmatrix}
                m_\epsi \\
                u_\epsi
        \end{bmatrix}
        ,
        \begin{bmatrix}
                m_\epsi \\
                u_\epsi
        \end{bmatrix}
        -  \begin{bmatrix}
                w \\
                v
        \end{bmatrix}
        \right\rangle 
        \leq 0,
\end{align}
which concludes the proof.
\end{proof}

\begin{remark}
\label{mcon}
By choosing $w=0$ and $w=2 m_\epsilon$ in \eqref{vi5-1} for $\lambda=1$,
we conclude that 
\begin{equation*}
                \int_{\Tt^d} \Big(\epsilon \big(m_\epsilon^2+ (\Delta^k m_\epsilon
\big)^2 + f_1[m_\epsilon, u_\epsilon]m_\epsilon\Big)\,\dx
        = 0.
        \end{equation*}
The preceding identity combined with \eqref{thebound} and   Assumptions~\ref{grH5-2} and~\ref{wcml1} yields
\begin{equation}
\label{eq:umabove}
\begin{aligned}
\int_{\Tt^d} u_\epsi m_\epsi\,\dx \leq C
\end{aligned}
\end{equation}
for some positive constant \(C\) independent of \(\epsi\). Moreover, under condition \eqref{eq:Hfromabove} in Assumption~\ref{asup}, we further obtain that %
\begin{equation}
\label{eq:umbelow}
\begin{aligned}
-\int_{\Tt^d} u_\epsi m_\epsi\,\dx \leq C,
\end{aligned}
\end{equation}
where \(C\) does not depend on \(\epsi\).

On the other hand, by choosing $v=1$ in \eqref{bf5-1l} for $\lambda=1$, we get 
\[
\int_{\Tt^d} m_{\epsilon}\,\dx=\int_{\Tt^d} \phi\,\dx-\epsilon \int_{\Tt^d} u_\epsilon\,\dx.
\]
Because of \eqref{thebound}, we have that $\epsilon \int_{\Tt^d} u_\epsilon\,\dx=O(\sqrt\epsilon)$.
Therefore, taking into account that $\int_{\Tt^d} \phi\,\dx=1$, we conclude
that
\begin{equation}
\label{eq:meanmepsi}
\begin{aligned}
\lim_{\epsilon\to0}\int_{\Tt^d} m_{\epsilon}\,\dx= 1.
\end{aligned}
\end{equation}
\end{remark}

Next, we prove Theorem~\ref{mt1} under the additional  Assumption \ref{asup}; that is, that there exists a weak solution,  \((m,u)
\in L^{q'}(\Tt^d;\Rr^+_0)\times W^{1,\alpha}(\Tt^d)\) with \(\int_{\Tt^d} m\,\dx =1 \),  to \eqref{mp}    provided that Assumptions~\ref{a0}--\ref{asup} hold.

\begin{proof}[Proof of Theorem~\ref{mt1} under Assumption \ref{asup}]
Let $(m_\epsilon, u_\epsilon)\in H^{2k}(\Tt^d)\times H^{2k}(\Tt^d),$ with
\(m_\epsi\geq 0\), be the
weak solution to \eqref{eq:wsepsi} given by Corollary \ref{cor49}.
We first establish some compactness properties of \((m_\epsilon, u_\epsilon)\).

From  \eqref{thebound},  Assumption~\ref{wcml1}, the condition \(m_\epsi\geq 0\), 
and extracting a subsequence if necessary, we can find $m\in L^1(\Tt^d)$
such that
\(m\geq 0\) and\begin{equation*}
m_\epsilon \rightharpoonup m \mbox{ in }L^1(\Tt^d). 
\end{equation*}
Moreover, assuming in addition Assumption \ref{asup}, 
\begin{equation}\label{wcm5B}
        m_\epsilon \rightharpoonup m \mbox{ in }L^{q'}(\Tt^d). 
\end{equation}

By combining Assumption \ref{a0} with \eqref{thebound}, we obtain further that 
\begin{equation}\label{ubdu5-1}
\int_{\Tt^d}|Du_\epsi|^\alpha\,\dx\leq C,
\end{equation}
where $C>0$ is uniformly bounded with respect to $\epsilon$.
Define 
\begin{equation*}
\tilde u_\epsilon
=
u_\epsilon - \int_{\Tt^d}u_\epsilon \,\dx.
\end{equation*}
Then, \eqref{ubdu5-1} and the Poincar\'{e}--Wirtinger inequality yield
\begin{equation*}
\|\tilde u_\epsilon\|_{L^\alpha(\Tt^d)}^\alpha
\leq
C\|D\tilde u_\epsilon\|_{L^\alpha(\Tt^d)}^\alpha
\leq 
C.
\end{equation*}
Furthermore, recalling Assumption \ref{asup},  $\tilde u_\epsilon$ converges, up to a subsequence, strongly in $L^q(\Tt^d)$ to some function $\tilde u$. 
Thus, using \eqref{wcm5B},%
\begin{equation}
\label{eq:limtum}
\begin{aligned}
\lim_{\epsi\to0}\int_{\Tt^d} \tilde u_\epsilon m_\epsilon\,\dx = \int_{\Tt^d} \tilde u m\,\dx.
\end{aligned}
\end{equation}
Additionally, by \eqref{eq:umabove} and \eqref{eq:umbelow}, we have that \[
\left|\int_{\Tt^d} u_\epsilon m_\epsilon \,\dx \right|
\]
is bounded uniformly in $\epsilon$. This uniform bound, the identity 
\[
\int_{\Tt^d} u_\epsilon m_\epsilon\,\dx=\left(\int_{\Tt^d} u_\epsilon\,\dx \right)\left(\int_{\Tt^d} m_\epsilon\,\dx \right)+\int_{\Tt^d} \tilde u_\epsilon m_\epsilon\,\dx,
\]
\eqref{eq:meanmepsi}, and \eqref{eq:limtum} imply that $\int_{\Tt^d} u_\epsilon $ converges, up to a subsequence. 
Therefore, $u_\epsilon$ converges to some function $u$ in $L^q(\Tt^d)$ and weakly in  $W^{1,\alpha}(\Tt^d)$.

Next, given \((w,v)\in H^{2k}(\Tt^d; \Rr_0^+)\times H^{2k}(\Tt^d)\), we observe that  \eqref{solmp5-1} and the monotonicity of $F_\epsilon$ (see Remark~\ref{rmk:notation})
 yield
\begin{equation}\label{eq5-1-1}
0\geq
\left\langle 
F_\epsilon
  \begin{bmatrix}
      m_\epsilon \\
      u_\epsilon
  \end{bmatrix}
,
  \begin{bmatrix}
      m_\epsilon \\
      u_\epsilon
  \end{bmatrix}
-  \begin{bmatrix}
      w \\
      v
  \end{bmatrix}
\right\rangle 
\geq
\left\langle 
F_\epsilon
  \begin{bmatrix}
      w \\
      v
  \end{bmatrix}
,
  \begin{bmatrix}
      m_\epsilon \\
      u_\epsilon
  \end{bmatrix}
-  \begin{bmatrix}
      w \\
      v
  \end{bmatrix}
\right\rangle 
=
\left\langle 
F
  \begin{bmatrix}
      w \\
      v
  \end{bmatrix}
,
  \begin{bmatrix}
      m_\epsilon \\
      u_\epsilon
  \end{bmatrix}
-  \begin{bmatrix}
      w \\
      v
  \end{bmatrix}
\right\rangle 
+
h_\epsilon,
\end{equation}
where
\begin{equation*}
h_\epsilon
=
\epsilon\int_{\Tt^d}\Big( w(m_\epsilon -w) +\Delta^k w\Delta^{k} (m_\epsilon
-w) 
+
 v(u_\epsilon -v) +  \Delta^k v\Delta^k (u_\epsilon -v)\Delta^k (u_\epsilon
-v)\Big)\,\dx.
\end{equation*}
From  \eqref{thebound}, we have uniform $L^2$-bounds
for $(\sqrt{\epsilon} m_\epsilon, \sqrt{\epsilon} \Delta^k m_\epsilon, \sqrt{\epsilon}
u_\epsilon,\sqrt{\epsilon}\Delta^k u_\epsilon)$; thus, using  H\"older's
inequality, we have
\begin{equation*}
\lim_{\epsilon \to 0}h_\epsilon=0.
\end{equation*}

Consequently, taking the limit in \eqref{eq5-1-1} as \(\epsi\to0\), we conclude
that
\begin{align*}
0\geq 
\left\langle 
F
  \begin{bmatrix}
      w \\
      v
  \end{bmatrix}
,
  \begin{bmatrix}
      m \\
      \tilde u
  \end{bmatrix}
-  \begin{bmatrix}
      w \\
      v
  \end{bmatrix}
\right\rangle 
\end{align*}
for all $(w,v)\in H^{2k}(\Tt^d; \Rr_0^+)\times H^{2k}(\Tt^d)$. Thus, \((m,u) \in L^{q'}(\Tt^d)\times W^{1,\alpha}(\Tt^d)\) is
 a weak solution to \eqref{mp} with \(\int_{\Tt^d} m\,\dx =1 \) by \eqref{eq:meanmepsi} and \eqref{wcm5B}.  
\end{proof}

\begin{remark}
        \label{412}
Without assuming Assumption \ref{asup},  we use the arguments in the preceding proof to show  the existence of a weak solution 
in a slightly different sense, under the assumptions of Theorem~\ref{mt1}. Precisely, there exists $(m, \tilde u)\in L^1(\Tt^d; \Rr_0^+)\times W^{1, \alpha}(\Tt^d)$ such that 
\eqref{ws}      holds for
 all $(w,v)\in \big\{w\in H^{2k}(\Tt^d; \Rr_0^+)| \, \int_{\Tt^d}w\,\dx=1\big\}\times H^{2k}(\Tt^d)$.
 
 Indeed, 
recalling that $f_2[w,v]=-w +\div(wD_p H(D v))-\phi$,  we have
 \begin{align*}
        \left\langle -f_2[w,v], u_\epsilon\right\rangle 
        &=
        \left\langle -f_2[w,v], \tilde u_\epsilon\right\rangle 
        +
        \int_{\Tt^d}u_\epsilon\,\dx \int_{\Tt^d}\big(w - \phi\big)\,\dx=
        \left\langle -f_2[w,v], \tilde u_\epsilon\right\rangle .
\end{align*}
Consequently, 
\begin{align*}
        \left\langle 
        F
        \begin{bmatrix}
                w \\
                v
        \end{bmatrix}
        ,
        \begin{bmatrix}
                m_\epsilon \\
                u_\epsilon
        \end{bmatrix}
        -  \begin{bmatrix}
                w \\
                v
        \end{bmatrix}
        \right\rangle 
        =
        \left\langle 
        F
        \begin{bmatrix}
                w \\
                v
        \end{bmatrix}
        ,
        \begin{bmatrix}
                m_\epsilon \\
                \tilde u_\epsilon
        \end{bmatrix}
        -  \begin{bmatrix}
                w \\
                v
        \end{bmatrix}
        \right\rangle .
\end{align*}
It suffices then to pass to the limit in the preceding expression.

\end{remark}


\subsection{Continuity method}\label{cm5}
Here, we use the continuity method to prove Theorem~\ref{mt1}.  
For that, we introduce the following \(\lambda\)-dependent regularized MFG:
\begin{equation}
\label{crp5-1}
\begin{aligned}
F_\epsilon^\lambda
  \begin{bmatrix}
      m \\
      u
  \end{bmatrix} = \begin{bmatrix}
      0 \\
      0
  \end{bmatrix},
\end{aligned}
\end{equation}
where $\lambda\in [0,1]$ and, for \((m,u) \in C^{4k}(\Tt^d) \times C^{4k}(\Tt^d) \) with \(m> 0\),
\begin{equation}\label{crp5-1a}
F_\epsilon^\lambda
  \begin{bmatrix}
      m \\
      u
  \end{bmatrix} =
\begin{bmatrix}- u-\lambda H(Du) + \lambda g(m) - \lambda V + p_\epsilon(m)+ \epsilon(m+\Delta^{2k}m )\\
m-\lambda\div\big(m D_p H(Du)\big) -\lambda \phi - (1-\lambda) + \epsilon(u+ \Delta^{2k} u) 
\end{bmatrix}.
\end{equation}

We want to show that there exists a solution to the preceding problem  when \(\lambda=1\),
which corresponds to the regularized MFG \eqref{mfg5-epsi} 
with $F_\epsilon$ given by \eqref{crp5-1x}.
 To this end, we set
\begin{equation}\label{defL5}
\begin{aligned}
\Lambda
=
\big\{\lambda \in [0,1]\ |\ 
&\mbox{problem \eqref{crp5-1} admits a smooth a solution, }(m_\lambda,u_\lambda)\in C^\infty(\Tt^d)\times C^\infty(\Tt^d), \\
& \mbox{with } \min_{\Tt^d} m_\lambda >0 
\big\},
\end{aligned}
\end{equation}
and start by proving that \(\Lambda\) is a nonempty set.

Fix $\epsilon \in (0,1)$ and let $\lambda=0$. For $c\in [0,\frac{1}{\epsilon})$, define
\begin{equation*}
i(c)=
-c+p_\epsilon(1-\epsilon c) +\epsilon(1-\epsilon c).
\end{equation*}
We have $i\in C([0,\frac{1}{\epsilon}))$,  $i(0)=\epsilon>0$, and $\lim_{c\to (\frac{1}{\epsilon})^-}i(c)=-\infty$.  Hence, there is $c^\ast\in [0,\frac{1}{\epsilon})$ such that $i(c^\ast)=0$. Therefore,
$(m, u)=(1-\epsilon c^\ast,c^\ast)$ is a solution to \eqref{crp5-1} with $\lambda=0$; note further  that $m=1-\epsilon c^\ast> 0$. Hence,  
\begin{equation}
\label{eq:notempty}
\begin{aligned}
\Lambda\neq \emptyset.
\end{aligned}
\end{equation}

Next, our goal is to prove that $\Lambda$ is both relatively open and closed, from which we deduce that
\(\Lambda=[0,1]\). In particular, there exists a solution for \(\lambda=1\) as stated in the next proposition.

\begin{pro}
\label{asol}
{
        Under Assumptions~\ref{a0}--\ref{grg5-1}},  there exists a smooth solution to \eqref{mfg5-epsi} with $F_\epsilon$  given by \eqref{crp5-1x}.
\end{pro}
\begin{proof}
As explained before, \eqref{mfg5-epsi} corresponds to \eqref{crp5-1} with $\lambda=1$. So the result follows from $\Lambda$ being non-empty 
by \eqref{eq:notempty}
and
 both relatively open and closed. The closedeness is proved in Lemma \ref{close5-1}, while the openess is proven in Lemma \ref{lem:open}.
\end{proof}

\subsubsection{\(\Lambda\) is closed}\label{close5}
The main result in  this subsection is the following lemma.
\begin{lemma}\label{close5-1}
{
Under Assumptions~\ref{a0}--\ref{grg5-1}},  the set $\Lambda$ introduced in \eqref{defL5} is closed.
\end{lemma}

The proof of Lemma~\ref{close5-1} relies on the establishing some a priori  uniform bounds   on the solutions of \eqref{crp5-1} as stated in the two following lemmas.

\begin{pro}\label{apestcm}
Let  \(\epsi\in(0,1)\) and \(\lambda\in[0,1]\), and assume that $(m_\lambda, u_\lambda)$, with \(m_\lambda > 0\), is a classical solution to \eqref{crp5-1}. 
Then, under {Assumptions~\ref{a0}--\ref{grg5-1}}, there exists a constant,
$C>0$, which is independent of $\epsilon$ and $\lambda$, such that $(m_\lambda,
u_\lambda)$ satisfies
\begin{equation}\label{apri-pterm}
\begin{split}
&\int_{\Tt^d}-p_\epsilon(m_\lambda)(\lambda \phi +1-\lambda)\,\dx  +\int_{\Tt^d}\big( p_\epsilon(m_\lambda)(m_\lambda-\epsilon)
+{\lambda}g(m_\lambda)m_\lambda
\big)\,\dx\\
&+\int_{\Tt^d}\big(
\lambda m_\lambda |D u_\lambda|^\alpha
+
\lambda (\lambda \phi +1-\lambda)|D u_\lambda|^\alpha
\big)\,\dx
+
\epsilon\int_{\Tt^d} (m_\lambda^2 + (\Delta^k m_\lambda)^2 + u_\lambda^2
+ (\Delta^k u_\lambda)^2)
\leq C.
\end{split}\end{equation}
Moreover, $(m_\lambda, u_\lambda)$ is  bounded in $H^{2k}(\Tt^d)\times
H^{2k}(\Tt^d)$, uniformly with respect to $\lambda$.
\end{pro}
%
%
\begin{proof}
Fix   \(\epsi\in(0,1)\) and \(\lambda\in[0,1]\), and assume that there exists $(m_\lambda, u_\lambda)
\in C^\infty(\Tt^d)\times C^\infty(\Tt^d)$ solving  \eqref{crp5-1} and such that $m_\lambda > 0$. 

We start multiplying the first equation
in \eqref{crp5-1} by $(m_\lambda-\epsilon -\lambda\phi - 1+\lambda)$ and
the second one by $u_\lambda$; we then  integrate
over $\Tt^d$ and add the resulting identities to  obtain
\begin{equation}
\begin{aligned}\label{eq:forapriori}
& \int_{\Tt^d}\Big[ \lambda m_\lambda\big( - H(D u_\lambda) +D_p H(D u_\lambda)\cdot D u_\lambda\big) + \lambda ( \epsilon + \lambda \phi + 1 - \lambda
) H(D u_\lambda)\Big]\,\dx\\
&+  \int_{\Tt^d} \Big[ \lambda  m_\lambda  g(m_\lambda)-(\lambda\phi
+ 1 - \lambda )
 p_\epsilon(m_\lambda)+(m_\lambda - \epsilon)
p_\epsilon(m_\lambda)\Big]\,\dx\\
&+
\epsilon\int_{\Tt^d} \Big[ m_\lambda^2 + (\Delta^k m_\lambda)^2 + u_\lambda^2
+ (\Delta^k u_\lambda)^2\Big]\,\\
&\quad = \int_{\Tt^d}\Big[
\lambda( \epsilon + \lambda \phi + 1 - \lambda ) g(m_\lambda) +
\lambda m_\lambda V(x) -\lambda( \epsilon + \lambda \phi + 1 - \lambda )V(x)
\Big]\,\dx\\
&\qquad +
\int_{\Tt^d} \Big[ \epsilon(\epsilon+\lambda \phi +1-\lambda )m_\lambda
+ \epsilon\lambda\Delta^k m_\lambda \Delta^k  \phi  \Big]\,\dx-\epsilon \int_{\Tt^d}  u_\lambda \, \dx.
\end{aligned}
\end{equation}

Next, we observe that \( 0\leq \epsilon + \lambda \phi +
1 - \lambda \leq 2+\Vert \phi\Vert_\infty\) because  $\epsilon\in(0,1)$, $\phi\geq
0$, and $\lambda \in [0,1]$. Thus, using Assumptions~\ref{a0}--\ref{grH5-2}  first and then Assumption~\ref{grg5-1} with \(\delta=\tfrac14\), we can estimate the first integral on the left-hand side of \eqref{eq:forapriori} as follows:
\begin{equation}
\label{eq:forapriori2}
\begin{aligned}
& \int_{\Tt^d}\Big[ \lambda m_\lambda\big( - H(D u_\lambda) +D_p H(D
u_\lambda)\cdot D u_\lambda\big) + \lambda ( \epsilon + \lambda \phi +
1 - \lambda
) H(D u_\lambda)\Big]\,\dx\\
& \quad \geq \frac1C \int_{\Tt^d}\Big[ \lambda m_\lambda|D u_\lambda|^\alpha + \lambda ( \epsilon + \lambda \phi +
1 - \lambda
)|D u_\lambda|^\alpha\Big]\,\dx -C\lambda\int_{\Tt^d} m_\lambda \,\dx +\lambda ( \epsilon + \lambda \phi +
1 - \lambda
)C \\
&\quad\geq \frac1C \int_{\Tt^d}\Big[ \lambda m_\lambda|D u_\lambda|^\alpha
+ \lambda ( \epsilon + \lambda \phi +
1 - \lambda
)|D u_\lambda|^\alpha\Big]\,\dx - \frac\lambda4\int_{\Tt^d} m_\lambda g(m_\lambda) \,\dx - \widetilde C,
\end{aligned}
\end{equation}
where \(\widetilde C\) is a positive constant that is independent of \(\epsi\) and of \(\lambda\). Similarly, the  right-hand side of \eqref{eq:forapriori} can be estimated as follows:
\begin{equation}
\label{eq:forapriori3}
\begin{aligned}
&\int_{\Tt^d}\Big[
\lambda( \epsilon + \lambda \phi + 1 - \lambda ) g(m_\lambda) +
\lambda m_\lambda V(x) -\lambda( \epsilon + \lambda \phi + 1 - \lambda )V(x)
\Big]\,\dx\\
&\quad +
\int_{\Tt^d} \Big[ \epsilon(\epsilon+\lambda \phi +1-\lambda )m_\lambda
+ \epsilon\lambda\Delta^k m_\lambda \Delta^k  \phi  \Big]\,\dx-\epsilon \int_{\Tt^d}  u_\lambda \, \dx\\
&\qquad\leq  \frac\lambda4\int_{\Tt^d} m_\lambda g(m_\lambda)
\,\dx +\frac\epsi2 \int_{\Tt^d}  \Big(m_\lambda^2 +(\Delta^k m_\lambda)^2+u_\lambda^2   \Big)  \,\dx + \widehat C, 
\end{aligned}
\end{equation}
where \(\widetilde C\) is a positive constant that is independent of \(\epsi\)
and of \(\lambda\). Then,  it follows from \eqref{eq:forapriori}--\eqref{eq:forapriori3} that
\begin{equation*}
\begin{aligned}
&\frac1C \int_{\Tt^d}\Big[ \lambda m_\lambda|D u_\lambda|^\alpha
+ \lambda ( \epsilon + \lambda \phi +
1 - \lambda
)|D u_\lambda|^\alpha\Big]\,\dx\\
&\quad+  \int_{\Tt^d} \Big[ \frac\lambda2  m_\lambda  g(m_\lambda)-(\lambda\phi
+ 1 - \lambda )
 p_\epsilon(m_\lambda)+(m_\lambda - \epsilon)
p_\epsilon(m_\lambda)\Big]\,\dx\\
&\quad+
\frac\epsilon2\int_{\Tt^d} \Big[ m_\lambda^2 + (\Delta^k m_\lambda)^2 + u_\lambda^2
+ (\Delta^k u_\lambda)^2\Big]\, \leq \widetilde C+ \widehat
C.
\end{aligned}
\end{equation*}

By  the properties of $p_\epsilon$, we have that $-p_\epsilon(m_\lambda)\geq
0$ and $(m_\lambda -\epsilon)p_\epsilon(m_\lambda)\geq 0$. Hence, 
\begin{equation*}\label{apri-crp5-1}
\epsilon\int_{\Tt^d} \big[m_\lambda^2 + (\Delta^k m_\lambda)^2 + u_\lambda^2
+ (\Delta^k u_\lambda)^2\big]\leq \overline C
\end{equation*}
for some positive and \(\lambda\)-independent constant $\overline C$. 
The Gagliardo--Nirenberg interpolation inequality then yields bounds  in $H^{2k}(\Tt^d)\times H^{2k}(\Tt^d)$  for $(m_\lambda, u_\lambda)$
that are  uniform with respect
to $\lambda$.
\end{proof}
\begin{remark}
\label{estimateforH}
By not using the bound by below for $H$ in Assumption \ref{grH5-2} in the estimate \eqref{eq:forapriori2}, we obtain when $\lambda=1$ that
\[
\int_{\Tt^d} H(Du)\,\dx\leq C,
\]
where the constant does not depend on $u$ or $\epsilon$.        
\end{remark}

Next, we prove a uniform bound on $m_\lambda$ from below.\begin{lem}\label{mlb5}
Let  $\Lambda$ be given by \eqref{defL5} and assume that {Assumptions~\ref{a0}--\ref{grg5-1}} hold. 
Then,
there exists a  constant, $c>0$, such that 
\begin{equation*}
\inf_{\lambda \in \Lambda}\inf_{x\in \Tt^d}m_\lambda (x) \geq c>0.
\end{equation*}
\end{lem} 

\begin{proof}
By contradiction, assume that $\inf_{\lambda\in \Lambda}\inf_{\Tt^d}m_\lambda
=0$. 

Let $\{\lambda_n\}_{n=1}^\infty\subset \Lambda$ be such that $c_n=\inf_{\Tt^d}m_{\lambda_n}\to
0$ as $n\to\infty$. 
By the definition of $\Lambda$,  there exists $x_n\in\Tt^d$ such that 
$c_n=\min_{\Tt^d}m_{\lambda_n}=m_{\lambda_n}(x_n)>0$.
Without loss
of generality, we may further assume that \(\lambda_n\to\lambda\) for some \(\lambda\in[0,1]\).
We have from Proposition~\ref{apestcm} and Morrey's
embedding theorem that
\begin{equation*}
R= \sup_{n\in\Nn}\|m_{\lambda_n}\|_{W^{1,\infty}(\Tt^d)}<\infty.
\end{equation*}

Thus, we can find \(n_0\in\Nn\) such that $0< c_n\leq \min\{\frac{\epsilon}{4},
\frac{\epsilon}{2\sqrt{d}R}\}$ for all $n\geq n_0$.
Hence, denoting by  $Q_n$  the cube in \(\Rr^d\) with side length $c_n$ and  centered at $x_n$, we have, for all \(x\in Q_n\) with   $n\geq n_0$, that
\begin{equation*}
0<m_{\lambda_n}(x)
\leq m_{\lambda_n}(x_n)+  
R|x-x_n|\leq c_n + \frac{\sqrt{d} c_n}{2}R \leq \frac{\epsilon}{2}.
\end{equation*}
In particular, recalling \eqref{eq:defpen},   $-p_\epsilon (m_{\lambda_n}(x) )=\tfrac{1}{(m_{\lambda_n}(x))^{d+1}}$ for all \(x\in Q_n\) with  $n\geq n_0$.

Then, 
using the fact that all integral terms on the left-hand side of \eqref{apri-pterm} are non-negative,  there exists a  constant, $C>0$, independent of \(n\) and for which we have, for all \(n\geq n_0\),
that\begin{align*}
C&\geq 
\int_{\Tt^d}-p_\epsilon(m_{{\lambda_n}})({\lambda_n} \phi + 1-{\lambda_n})\,\dx
\geq
\int_{Q_n}\frac{1}{m_{{\lambda_n}}^{d+1}}({\lambda_n} \phi + 1-{\lambda_n})\,\dx\\
&\geq
\frac{c_n^{d} }{\big(\frac{\sqrt{d} c_n}{2}R+c_n\big)^{d+1}}\big( {\lambda_n}
\phi_{min} + 1-{\lambda_n}\big)=
\Big(\frac{2}{\sqrt{d}R+2}\Big)^{d+1}\frac{1}{c_n}\big( {\lambda_n}
\phi_{min} + 1-{\lambda_n} \big), 
\end{align*}
where \(\phi_{min}=\min_{x\in\Tt^d}\phi(x)>0\) by Assumption \ref{a0}. 
Finally, we observe that \(\lim_{n\to\infty} \tfrac{1}{c_n} = +\infty\) and \(\lim_{n\to\infty} ({\lambda_n}
\phi_{min} + 1-{\lambda_n}) = {\lambda}
\phi_{min} + 1-{\lambda}>0\) because \(\lambda\in[0,1]\) and \(\phi_{min}>0\). Thus, taking the limit as  $n \to \infty$ in the preceding estimate leads to a contradiction, which concludes the proof of the lemma.   
\end{proof}

\begin{proof}[Proof of Lemma~\ref{close5-1}]
Let $\{\lambda_n\}_{n=1}^\infty \subset \Lambda$
be such that $\lambda_n \to \lambda $ as $n\to \infty$. Then, \(\lambda\in[0,1]\). To show that \(\lambda\in\Lambda\), we use the definition of $\Lambda$  and  Lemma~\ref{mlb5} to find, for each $n\in\Nn$,
 a pair $(m_n, u_n)\in C^\infty(\Tt^d)\times C^\infty(\Tt^d)$ with $\min_{\Tt^d}m_n\geq  c>
0$ solving \eqref{crp5-1} with $\lambda$ replaced by $\lambda_n$, where the constant \(c\) does not depend on \(n\). 

From Proposition~\ref{apestcm},  there exists a subsequence of $\{(m_{n}, u_{n})\}_{n=1}^\infty$, $\{(m_{n_j}, u_{n_j})\}_{j=1}^\infty$,   and $(m,u)\in H^{2k}(\Tt^d)\times H^{2k}(\Tt^d)$
such that $(m_{n_j}, u_{n_j})\rightharpoonup (m, u)$ in $H^{2k}(\Tt^d)\times
H^{2k}(\Tt^d)$. 
Using the Rellich--Kondrachov theorem and Morrey's theorem, $(m_{n_j}, u_{n_j})\to
(m, u)$ in $C^{2,l}(\Tt^d)\times C^{2,l}(\Tt^d)$ for some
$l\in (0,1)$.  Thus,
we have $\min_{\Tt^d}m>0$.

On the other hand,  for all $(w,v)\in H^{2k}(\Tt^d)\times
H^{2k}(\Tt^d)$, we obtain from  \eqref{crp5-1} that
\begin{align*}
\int_{\Tt^d}\Big(
\epsilon (m_{n_j} w + \Delta^k m_{n_j} \Delta^k w)+ g_1[m_{n_j}, u_{n_j}]w
\Big)\,\dx=0,\\
\int_{\Tt^d}\Big(
\epsilon(u_{n_j} v + \Delta^k u_{n_j} \Delta^k v)+ g_2[m_{n_j}, u_{n_j}]v
\Big)\,\dx=0,
\end{align*}
where 
$g_1[m_{n_j}, u_{n_j}]= -u_{n_j} -\lambda_{n_j} H(D u_{n_j}) +\lambda_{n_j}
g(m_{n_j}) - \lambda_{n_j} V(x) + p_\epsilon(m_{n_j})$ 
and 
$g_2[m_{n_j}, u_{n_j}]=m_{n_j}-\lambda_{n_j}\div(m_{n_j}D_p H(D u_{n_j}))-\lambda_{n_j}
\phi  -(1-\lambda_{n_j})$. 
Using the weak convergence in $H^{2k}(\Tt^d)\times H^{2k}(\Tt^d) $ and the
strong convergence in $C^{2,l}(\Tt^d) \times C^{2,l}(\Tt^d) $ of $\{(m_{n_j}, u_{n_j})\}_{j=1}^\infty$ to \((m,u)\) justified above, 
and the convergence $\lambda_{n_j} \to \lambda\in [0,1]$,
taking a limit as $j\to \infty$ in the previous identities  yields\begin{align*}
\int_{\Tt^d}\Big(
\epsilon (m w + \Delta^k m \Delta^k w)+ g_1[m,
u]w
\Big)\,\dx=0,\\
\int_{\Tt^d}\Big(
\epsilon(u v + \Delta^k u \Delta^k v)+ g_2[m, u]v
\Big)\,\dx=0.
\end{align*}
Because $g_1[m, u]$, $g_2[m, u]\in C^{0,l}(\Tt^d)$
for some $l\in (0,1)$,  we have $(m,u)
\in C^{4k,l}(\Tt^d) \times  C^{4k,l}(\Tt^d)$ by standard elliptic-regularity arguments, and \((m,u)\) solves \eqref{crp5-1}. Finally, using a bootstrap argument, we conclude that $(m,u) \in C^{\infty}(\Tt^d) \times C^{\infty}(\Tt^d)$.  Hence, $\lambda \in \Lambda$,   which proves that
 $\Lambda$ is closed.
\end{proof}

\subsubsection{\(\Lambda\) is open}\label{open5}
Next, we show that $\Lambda$ is relatively open.
\begin{lemma}\label{lem:open}
Under {Assumptions~\ref{a0}--\ref{grg5-1}}, the set $\Lambda$ introduced
in \eqref{defL5} is relatively open in  \([0,1]\).
\end{lemma}
\begin{proof}
Fix $\lambda_0 \in \Lambda $ and let $(m_0,u_0)$ be a smooth solution to \eqref{crp5-1} with \(\lambda=\lambda_0\) such that $\inf_{\Tt^d}m_0>0$. We will find a neighborhood of \(\lambda_0\) in \([0,1]\) contained in \(\Lambda\) using the implicit function theorem in Banach spaces (see, for example, \cite{Die}).
 Here, we use the implicit function theorem in $C^{4k}(\Tt^d)\times C^{4k}(\Tt^d)$.  That is, we regard the map $F$ as a map from $C^{4k}(\Tt^d)\times C^{4k}(\Tt^d)$ to $C(\Tt^d)\times C(\Tt^d)$.

First, recalling \eqref{crp5-1a}, we compute  the Fr\'{e}chet derivative, \(\Ll\),  of the mapping  $(m,u) \mapsto F_\epsilon^{\lambda_0}(m,u)$  at \((m_0, u_0)\). For each $(w,v)\in C^{4k}(\Tt^d)\times C^{4k}(\Tt^d)$, we have
\begin{equation*}
\Ll 
  \begin{bmatrix}
      w\\
      v
  \end{bmatrix}
=
  \begin{bmatrix}
      -v - \lambda_0 D_p H(D u_0)\cdot D v + \lambda_0 g'(m_0) w + p_\epsilon'(m_0)w + \epsilon(w + \Delta^{2k}w) \\
      w - \lambda_0 \div\big(w D_p H(D u_0)+  m_0 D_{pp}^2H(Du_0)D v\big) + \epsilon(v + \Delta^{2k}v) 
  \end{bmatrix}.
\end{equation*}
We want to prove that $\Ll$ is an isomorphism from $C^{4k}(\Tt^d)\times C^{4k}(\Tt^d)$ to $C(\Tt^d)\times C(\Tt^d)$.

Because $\Ll$ is linear, 
to show that  $\Ll$ is injective is equivalent to showing that its kernel is trivial. 
For that,
 we assume that $(w, v) \in C^{4k}(\Tt^d)\times C^{4k}(\Tt^d)$ satisfies $\Ll[w,v]=0$. Recalling that $H$ is convex, $g$ and $p_\epsilon$ are increasing, and $m_0 >0$, it follows that \begin{equation}\label{ineq5-1}
\lambda_0 m_0 (Dv)^TD_{pp}^2H(D u_0)D v\geq 0,\ \ 
\lambda_0 g'(m_0)w^2\geq 0,\ \ 
p_\epsilon'(m_0)w^2\geq 0.
\end{equation}
Therefore, taking the inner product between $\Ll[w,v]$ and $(w, v)$, an integration by parts yields 
\begin{align*}
0&=
\int_{\Tt^d}\Big(
\lambda_0  \big(g'(m_0 )w^2 +m_0  (Dv)^TD_{pp}^2H(D u_0)D v
\big) + p_\epsilon'(m_0 )w^2\Big)\,\dx\\
&\quad+ \epsilon\int_{\Tt^d}\Big(
w^2 + (\Delta^{k}w)^2
+
v^2 + (\Delta^{k}v)^2
\Big)\,\dx\\
&\geq
\epsilon\int_{\Tt^d}\Big(
w^2 + (\Delta^{k}w)^2
+
v^2 + (\Delta^{k}v)^2
\Big)\,\dx\geq 0,
\end{align*}
from which we obtain that $(w, v) \equiv (0, 0)$.

Next, we show that $\Ll$ is surjective. For this, we need to show that given $(h_1, h_2) \in C(\Tt^d) \times C(\Tt^d)$
there exists $(\bar w,\bar v)\in C^{4k}(\Tt^d)\times C^{4k}(\Tt^d)$
that solves $\Ll(\bar w,\bar v)=(h_1, h_2)$. As a first step, we a find a solution to this problem in 
$H^{2k}(\Tt^d)\times H^{2k}(\Tt^d)$.

For $(w_1,v_1)$, $(w_2, v_2)\in H^{2k}(\Tt^d)\times H^{2k}(\Tt^d)$, define 
\begin{align*}
B\left[
  \begin{bmatrix}
      w_1\\
      v_1
  \end{bmatrix},
  \begin{bmatrix}
      w_2\\
      v_2
  \end{bmatrix}
\right]=&
\int_{\Tt^d}\Big(
      -v_1w_2 - \lambda_0 D_p H(D u_0)\cdot D v_1 w_2 + \lambda_0 g'(m_0) w_1 w_2 + p_\epsilon'(m_0)w_1w_2\\
      &\qquad\enspace
      +w_1 v_2 - \lambda_0 \div\big(w_1 D_p H(D u_0)+  m_0 D_{pp}^2H(Du_0)D v_1\big)v_2
\Big)\,\dx\\
&+
\epsilon\int_{\Tt^d}(
w_1w_2 + \Delta^k w_1 \Delta^k w_2
+
v_1v_2 + \Delta^k v_1 \Delta^k v_2
)\,\dx.
\end{align*}
We consider the problem of finding $(\bar w,\bar v)\in H^{2k}(\Tt^d)\times H^{2k}(\Tt^d)$ such that, for all $(\xi, \eta)\in H^{2k}(\Tt^d)\times H^{2k}(\Tt^d)$,
we have\begin{equation}\label{bfmfg5}
B \left[ \begin{bmatrix}
      \bar w\\
      \bar v
  \end{bmatrix},
  \begin{bmatrix}
      \xi\\
      \eta
  \end{bmatrix}\right]
=
\left\langle 
  \begin{bmatrix}
      h_1\\
      h_2
  \end{bmatrix},
  \begin{bmatrix}
      \xi\\
      \eta
  \end{bmatrix}
\right\rangle .
\end{equation}
Given  $(w,v)\in H^{2k}(\Tt^d)\times
H^{2k}(\Tt^d)$, an integration by parts and \eqref{ineq5-1}  yield\begin{equation*}
B\left[ \begin{bmatrix}
      w\\
      v
  \end{bmatrix},
  \begin{bmatrix}
      w\\
      v
  \end{bmatrix}\right]
\geq
\epsilon(\|w\|_{L^2(\Tt^d)}^2+\|v\|_{L^2(\Tt^d)}^2+\|\Delta^kw\|_{L^2(\Tt^d)}^2+\|\Delta^kv\|_{L^2(\Tt^d)}^2).
\end{equation*}
Thus, by the Gagliardo--Nirenberg interpolation inequality,  there exists a constant, $C_\epsilon>0$, independent of \((w,v)\), for which
\begin{equation*}
B\left[ \begin{bmatrix}
      w\\
      v
  \end{bmatrix},
  \begin{bmatrix}
      w\\
      v
  \end{bmatrix}\right]
\geq
C_\epsilon\big(\|w\|_{H^{2k}(\Tt^d)}^2 + \|v\|_{H^{2k}(\Tt^d)}^2\big).
\end{equation*}
On the other hand, the smoothness of  $(m_0,u_0)$ and  the data, together with the fact that  $\min_{\Tt^d} m_0>0$, yields a constant, $C>0$, such that for all $(w_1,v_1)$, $(w_2, v_2)\in H^{2k}(\Tt^d)\times
H^{2k}(\Tt^d)$, we have
\begin{equation*}
\bigg|
B\left[
  \begin{bmatrix}
      w_1\\
      v_1
  \end{bmatrix},
  \begin{bmatrix}
      w_2\\
      v_2
  \end{bmatrix}
\right]\bigg|
\leq
C(\|w_1\|_{H^{2k}(\Tt^d)}+\|v_1\|_{H^{2k}(\Tt^d)})(\|w_2\|_{H^{2k}(\Tt^d)}+\|v_2\|_{H^{2k}(\Tt^d)}).
\end{equation*}
Hence, applying the Lax--Milgram theorem, there exists $(\bar w,\bar v)\in H^{2k}(\Tt^d)\times H^{2k}(\Tt^d)$ satisfying \eqref{bfmfg5}. Finally, using  elliptic-regularity  and  bootstrap arguments, we deduce that $(\bar w,\bar v)\in C^{4k}(\Tt^d)\times C^{4k}(\Tt^d)$, 
 which, together with \eqref{bfmfg5}, shows that $\Ll$ is surjective. 

Therefore, by the implicit function theorem, $\Lambda$ is relatively open in \([0,1]\).
\end{proof}

\subsubsection{Proof of Theorem \ref{mt1}}

Now, we prove the existence of weak solutions. 
\begin{proof}[Proof of Theorem~\ref{mt1}]
Fix $\epsilon \in (0,1)$ and let   $(m_\epsilon,u_\epsilon)\in C^{\infty}(\Tt^d)\times C^{\infty}(\Tt^d)$
be the solution of  \eqref{mfg5-epsi} (for  $F_\epsilon$ given by \eqref{crp5-1x}) constructed in Proposition \ref{asol}. 
Recall further that  $\min_{\Tt^d} m_\epsilon> 0$ by Lemma \eqref{mlb5}. Next, we establish some uniform estimates in \(\epsi\) for \((m_\epsilon,u_\epsilon)\).

  By \eqref{apri-pterm} with $\lambda=1$, and recalling that \(\min_{\Tt^d} \phi >0\) by Assumption~\ref{a0} and \(p_\epsi \leq 0\) by construction, we have that
\begin{equation}\label{apri5-2}
\int_{\Tt^d} (m_\epsilon g(m_\epsilon) + m_\epsilon|Du_\epsilon|^\alpha + |Du_\epsilon|^\alpha - p_\epsilon(m_\epsilon))\,\dx
+
\epsilon\int_{\Tt^d} \big(m_\epsilon^2 + (\Delta^k m_\epsilon)^2 + u_\epsilon^2 + (\Delta^k u_\epsilon)^2\big)\,\dx\leq C,
\end{equation}
where $C$ is independent of $\epsilon$ and $(m_\epsilon,u_\epsilon)$. On the other hand, integrating the first equation in \eqref{mfg5-epsi} over $\Tt^d$, 
we get from the conditions \(g\geq 0\) and \(p_\epsi \leq 0\) that %
\begin{align*}
\Big|\int_{\Tt^d} u_\epsilon\,\dx\Big|
\leq 
\Big|\int_{\Tt^d}H(Du_\epsilon)\,\dx \Big| + 
\int_{\Tt^d}\big(g(m_\epsilon) + |V(x)|  -p_\epsilon(m_\epsilon)+ \epsilon m_\epsilon\big)\,\dx.
\end{align*}
Then, using 
Remark \ref{estimateforH}, Assumptions~\ref{grH5-2} and~\ref{grg5-1}, and \eqref{apri5-2}, we conclude that
\begin{equation}\label{bintu5-1}
\Big|\int_{\Tt^d}u_\epsilon\,\dx\Big|
\leq
C,
\end{equation}
where $C>0$ is independent of $\epsilon$ and $(m_\epsilon,u_\epsilon)$.
The Poincar\'{e}--Wirtinger inequality and \eqref{apri5-2}--\eqref{bintu5-1} yield 
\begin{equation*}
\|u_\epsilon\|_{L^\alpha(\Tt^d)}
\leq
\Big\| u_\epsilon - \int_{\Tt^d}u_\epsilon\,\dx\Big\|_{L^\alpha(\Tt^d)}+C
\leq
C\|D u_\epsilon\|_{L^\alpha(\Tt^d)} +C 
\leq C.
\end{equation*}
The preceding estimate and \eqref{apri5-2} show that $u_\epsilon$ is  bounded in $W^{1,\alpha}(\Tt^d)$, uniformly with respect to $\epsilon$.

Using the monoticity of $F_\epsilon$ addressed in Remark \ref{remfemon} and the definition of $F$ in \eqref{defF1}, for all $(w,v)\in H^{2k}(\Tt^d; \Rr^+)\times  H^{2k}(\Tt^d)$, we have
\begin{equation}\label{eq:almweasol}
0= 
\left\langle 
F_\epsilon
  \begin{bmatrix}
      m_\epsilon\\
      u_\epsilon
  \end{bmatrix}
,
  \begin{bmatrix}
      m_\epsilon \\
      u_\epsilon
  \end{bmatrix}
-  \begin{bmatrix}
      w \\
      v
  \end{bmatrix}
\right\rangle 
\geq
\left\langle 
F_\epsilon
  \begin{bmatrix}
      w \\
      v
  \end{bmatrix}
,
  \begin{bmatrix}
      m_\epsilon \\
      u_\epsilon
  \end{bmatrix}
-  \begin{bmatrix}
      w \\
      v
  \end{bmatrix}
\right\rangle 
=
\left\langle 
F
 \begin{bmatrix}
      w \\
      v
  \end{bmatrix}
,
  \begin{bmatrix}
      m_\epsilon \\
      u_\epsilon 
  \end{bmatrix}
-  \begin{bmatrix}
      w \\
      v
  \end{bmatrix}
\right\rangle 
+
c_\epsilon,
\end{equation}
where 
\begin{align*}
c_\epsilon=\int_{\Tt^d}
\epsilon\big(
 w(m_\epsilon-w) + \Delta^k w \Delta^k(m- w)
+
v(u_\epsilon -v) + \Delta^k v \Delta^k (u_\epsilon-v)\big)+ p_\epsilon(w)(m_\epsilon-w)\Big)\,\dx.
\end{align*} 

Next, we observe that, extracting a subsequence if necessary, the uniform estimates in  \(\epsi\)  of \(\Vert u_\epsi\Vert_{W^{1,\alpha}(\Tt^d)}\)  proved above and Assumption~\ref{wcml1} combined with  \eqref{apri5-2} yield
\begin{align*}
&m_\epsilon \rightharpoonup m\ \mbox{in } L^1(\Tt^d),\\
&u_\epsilon \rightharpoonup u\ \mbox{in } W^{1,\alpha}(\Tt^d),\\
&\sqrt{\epsilon}(\sqrt{\epsilon}m_\epsilon), \sqrt{\epsilon}(\sqrt{\epsilon}u_\epsilon)
\to 0\ \mbox{in }H^{2k}(\Tt^d),
\end{align*}
for some \(m\in L^1(\Tt^d)\), with \(m \geq 0\), and \(u\in W^{1,\alpha}(\Tt^d) \). Because  $p_\epsilon(w)\to 0$ in $C(\Tt^d)$ whenever  $w\in H^{2k}(\Tt^d; \Rr^+)$, we conclude that
\begin{equation*}
 \lim_{\epsilon\to0}c_\epsilon=0
\end{equation*}
for  each $(w,v)\in H^{2k}(\Tt^d; \Rr^+)\times H^{2k}(\Tt^d)$. 
Therefore,  arguing as  in Subsection~\ref{vpbfa5-1}, we can let \(\epsi\to0\) in \eqref{eq:almweasol} to conclude that
\begin{equation*}
0\geq
\lim_{\epsilon \to 0}\left(
\left\langle 
F
  \begin{bmatrix}
      w \\
      v
  \end{bmatrix}
,
  \begin{bmatrix}
      m_\epsilon \\
      u_\epsilon
  \end{bmatrix}
-  \begin{bmatrix}
      w \\
      v
  \end{bmatrix}
\right\rangle 
+c_\epsilon\right)
=
\left\langle 
F
  \begin{bmatrix}
      w \\
      v
  \end{bmatrix}
,
  \begin{bmatrix}
      m \\
      u
  \end{bmatrix}
-  \begin{bmatrix}
      w \\
      v
  \end{bmatrix}
\right\rangle .
\end{equation*}
Finally, we observe that Assumption~\ref{a0}  allows us to show the non-positivity of the right-hand side of the preceding estimate for all \((w,v)\in H^{2k}(\Tt^d; \Rr^+_0)\times H^{2k}(\Tt^d)\).  Thus, \((m,u)\in L^1(\Tt^d) \times W^{1,\alpha}(\Tt^d)\) is  a weak solution to \eqref{mfg5-1}.
\end{proof}

\section{Existence of solutions for time-dependent MFGs}
\label{t-dep5}
In this last section, we explain briefly how the prior techniques based on monotonicity can be used to study the time-dependent MFG in 
 \eqref{ext1}--\eqref{ext1bc}. Many of the steps are similar to the stationary case, so, here, we highlight only the main points and refer the interested reader to \cite{FeGoTa21}  for a detailed analysis.

Set $\Omega_T=(0,T)\times \Tt^d$ and fix  $m_0$, $u_T\in C^{4k}(\Tt^d)$ such that $\min_{x\in\Tt^d}m_0(x)>0$ and $\int_{\Tt^d}m_0 (x)\,\dx=1$, where   $k\in \Nn$ is such
that $2k\geq \frac{d+1}{2}+4$.
Consider the following (elliptic)  regularized version of the MFG in \eqref{ext1}:
\begin{equation}\label{extr1}
\begin{cases}
 u_t + \Delta u - H(Du) + g(m) + V + \epsilon\big(m+\sum_{|\beta|=2k} \partial_{t,x}^{2\beta}m\big)=0& \mbox{in}\ \Omega_T,\\
m_t - \Delta m - \div(mD_p H(Du)) + \epsilon\big(u+ \sum_{|\beta|=2k} \partial_{t,x}^{2\beta} u\big) =0& \mbox{in}\ \Omega_T,\\
m(0,\cdot)=m_0,\ u(T,\cdot)=u_T& \mbox{in}\ \Tt^d,
\end{cases}
\end{equation}
where 
\begin{equation*}
\begin{aligned}
\partial_{t,x}^\beta= \frac{\partial^{|\beta|}}{\partial t^{\beta_0}\partial x_1^{\beta_1}...\partial x_d^{\beta_d}},
\quad \beta=(\beta_0,\beta_1,...,\beta_d)\in\Nn_0^{d+1}, \quad |\beta|
= \sum_{i=0}^d \beta_i. \end{aligned}
\end{equation*}
The presence of high-order derivatives in time requires the following additional boundary conditions. For each $i\in \Nn$ such that $2\leq i\leq 2k$,
\begin{equation*}
\begin{cases}
\sum_{j=1}^{2k}\partial_t^{2j-1}(M_j m)=0  &\mbox{on } \{T\}\times \Tt^d,\\
\sum_{j=1}^{2k}\partial_t^{2j-1}(M_j u)=0  &\mbox{on } \{0\}\times \Tt^d,\\
\sum_{j=i}^{2k}\partial_t^{2j-1}(M_j m)=0 &\mbox{on } \{0,T\}\times \Tt^d,\\
\sum_{j=i}^{2k}\partial_t^{2j-1}(M_j u)=0&\mbox{on } \{0,T\}\times \Tt^d,
\end{cases}
\end{equation*}
where $M_j= \sum_{|\gamma|=2k-j}\partial_x^{2\gamma}$. These boundary conditions are crucial both to preserve the monotonicity
of the original problem  and enable a priori estimates for classical and weak solutions of \eqref{extr1}.

 We show how
to address \eqref{extr1} using similar techniques to those in Subsection~\ref{vpbfa5-1} under Assumptions~\ref{a0}--\ref{wcml1}
(with the obvious modifications to incorporate the dependence
in \(t\)). 

Set
\begin{equation}\label{eq:defsets}
\begin{aligned}
\Aa&=
\left\{m\in H^{2k}(\Omega_T)\ |\ m(0,x)=m_0(x),\ m\geq 0\right\},\\
\Bb&=
\left\{u\in H^{2k}(\Omega_T)\ |\ u(T,x)=u_T\right\}.
\end{aligned}
\end{equation}
We prove the existence of a pair 
$(m,u)\in \Aa\times \Bb$ that is a solution to the variational inequality associated with \eqref{extr1}. In particular, this pair satisfies
\begin{align}
&\int_0^T\!\!\!\int_{\Tt^d}\left(
 u_t + \Delta u - H(Du) + g(m) + V 
\right)(w-m)\,\dx\dt\nonumber\\
&\qquad+
\epsilon\int_0^T\!\!\!\int_{\Tt^d}\bigg[m (w-m)
+\sum_{|\beta|=2k}\partial_{t,x}^\beta m \partial_{t,x}^\beta(w-m)
\bigg]\,\dx\dt\geq0,\label{wsmt1}\\
&\int_0^T\!\!\!\int_{\Tt^d}\left(
m_t -\Delta m -\div\big(mD_pH(x,Du)\big)
\right (v-u)\,\dx\dt\nonumber\\
&\qquad+
\epsilon\int_0^T\!\!\!\int_{\Tt^d}\bigg[
u (v-u)
+
\sum_{|\beta|=2k}\partial_{t,x}^\beta u \partial_{t,x}^\beta
(v-u)
\bigg]\,\dx\dt=0\label{wsut1}
\end{align}
for all $(w, v)\in \Aa\times \Bb$.

\subsection{A priori estimates}
To prove a similar estimate to that in \eqref{apriori5-1}, choose $w=m_0\in \Aa$ in \eqref{wsmt1} and $v=u_T\in \Bb$ in \eqref{wsut1}. Then, adding the resulting inequalities, integrating by parts, and using Assumption~\ref{grg5-1},    we conclude that
\begin{align*}
&\int_0^T\!\!\!\int_{\Tt^d} \big(m g(m) + m |Du|^\alpha+m_0|Du|^\alpha\big)\,\dx\dt\\
&\qquad+
\epsilon\int_0^T\!\!\!\int_{\Tt^d} \bigg(m^2 + \sum_{|\beta|=2k}(\partial_{t,x}^\beta m)^2 + u^2 + \sum_{|\beta|=2k}(\partial_{t,x}^\beta u)^2\bigg)\,\dx\dt
\leq  C \big(\|m\|_{L^1(\Omega_T)}+ \|Du\|_{L^1(\Omega_T)}+1\big).
\end{align*}

By Assumption~\ref{grg5-1}  and the fact that $\min_{x\in\Tt^d}m_0(x)\geq c>0$ for some $c>0$, 
 we can find  for each  $\delta \in(0,1)$ a constant, $C_\delta$, independent of $\epsilon$ and $(m,u)$, for which we  have
\begin{align*}
\int_0^T\!\!\!\int_{\Tt^d}(m + |Du|)\,\dx\dt
\leq 
\delta\Big(\int_0^T\!\!\!\int_{\Tt^d}(mg(m) + |Du|^\alpha)\,\dx\dt\Big) + C_\delta.
\end{align*}
Therefore, choosing $\delta\in(0,1)$ small enough, there is a  constant, $C>0$, that is independent of $\epsilon$ and $(m,u)$, such that
\begin{equation}\label{apritr1}
 \int_0^T\!\!\!\int_{\Tt^d}(mg(m)+|Du|^\alpha)\,\dx\dt
+
\epsilon\int_0^T\!\!\!\int_{\Tt^d} \Big[m^2 + \sum_{|\beta|=2k}(\partial_{t,x}^\beta m)^2 + u^2 + \sum_{|\beta|=2k}(\partial_{t,x}^\beta u)^2\Big]\,\dx\dt\leq C.
\end{equation}

\subsection{Variational problem and bilinear form approach}\label{vpbftr1}

Let
\[
\begin{aligned}
        \tilde \Aa&=
        \left\{\tilde m\in H^{2k}(\Omega_T)\ |\ \tilde m(0,x)=0,\ \tilde m+m_0\geq 0\right\},\\
        \tilde \Bb&=
        \left\{\tilde u\in H^{2k}(\Omega_T)\ |\ \tilde u(T,x)=0\right\}.
\end{aligned}
\]
For $(m,u)\in \Aa \times \Bb$, we write $u=u_T+\tilde u$ and $m=m_0+\tilde m$ with $(\tilde m,\tilde u)\in \tilde \Aa \times \tilde \Bb$.
Accordingly, we modify $H$, $g$, and $V$, as follows:
\begin{align*}
&\tilde H(D \tilde u)= H( D u_T+D \tilde u),\ \ 
\tilde g(\tilde m)= g(m_0+\tilde m),\ \ 
\tilde V(t,x)= V(t,x)+\Delta u_T.
\end{align*}
Set
\begin{align*}
f_1[\tilde m,\tilde u]&=  \tilde u_t + \Delta \tilde u- \tilde H(D \tilde u) + \tilde g(\tilde m) + \tilde V,\\
f_2[\tilde m,\tilde u]&= \tilde m_t - \Delta \tilde m -\Delta m_0- \div((m_0+\tilde m)D_p \tilde H(D\tilde u)).
\end{align*}

Fix $(\tilde m_1,\tilde u_1)\in H^{2k-1}(\Omega_T)\times H^{2k-1}(\Omega_T)$,
and  for $w$, $v_1$, $v_2\in H^{2k}(\Omega_T)$, define
\begin{align*}
J(w)
&=
\int_0^T\!\!\!\int_{\Tt^d}\bigg[\frac{\epsilon}{2}\bigg( w^2+ \sum_{|\beta|=2k}(\partial_{t,x}^\beta w )^2\bigg)+f_1[\tilde m_1,\tilde u_1]w\bigg]\,\dx\dt,\\
B[v_1,v_2]
&=
\epsilon\int_0^T\!\!\!\int_{\Tt^d}\Big[ v_1v_2 + \sum_{|\beta|=2k}\partial_{t,x}^\beta v_1\partial_{t,x}^\beta v_2\Big]\,\dx\dt.
\end{align*}
As in Subsection~\ref{vpbfa5-1},
we search for a pair, $(\tilde m,\tilde u)\in \tilde \Aa\times \tilde \Bb$, satisfying
\begin{equation}\label{vpt1}
J(\tilde m)=\inf_{w\in \tilde \Aa}J(w)
\end{equation}
and
\begin{equation}\label{bft1}
B[\tilde u,v]=\left\langle f_2[\tilde m_1,\tilde u_1],v\right\rangle \ \ \mbox{for all } v\in \tilde \Bb.
\end{equation}
It can be checked that   similar arguments to those   of Subsection~\ref{vpbfa5-1}
allow us to prove that such   a pair  $(\tilde m,\tilde u)\in \tilde \Aa\times \tilde \Bb$ exists, being   $\tilde m$ the  unique minimizer to \eqref{vpt1} and $\tilde u$ the unique weak solution to \eqref{bft1}, respectively.

\subsection{Existence of a fixed-point}\label{fpttr1}
Next, we construct a solution, $(\tilde m,\tilde u)$, to  \eqref{wsmt1}--\eqref{wsut1}. As in the previous section, we use 
Theorem~\ref{Sch4-1}. Let $\widehat \Aa$ and $\widehat \Bb$ be given by
\begin{align*}
\widehat \Aa
&=
\left\{\tilde m\in H^{2k-1}(\Omega_T)\ |\ \tilde m(0,x)=0,\ \tilde m+m_0\geq 0\right\},\\
\widehat \Bb
&=
\left\{\tilde u\in H^{2k-1}(\Omega_T)\ |\ \tilde u(T,x)=0\right\}.
\end{align*}
Define $A:\widehat \Aa\times \widehat \Bb\to \widehat \Aa\times \widehat \Bb$ by
\begin{align*}
\begin{bmatrix}
      \tilde m_2 \\
      \tilde u_2
\end{bmatrix}
=
A
\begin{bmatrix}
      \tilde m_1 \\
      \tilde u_1
\end{bmatrix},
\end{align*}
where $(\tilde m_2,\tilde u_2)$ satisfies \eqref{vpt1} and \eqref{bft1}, respectively. Using a similar argument to the one in Subsection~\ref{vpbfa5-1} combined with the a priori estimates \eqref{apritr1}, 
we have that $A$ is continuous and compact.

Next, let
\begin{align*}
S=\left\{
(\tilde m,\tilde u)\in \widehat \Aa\times \widehat \Bb\ \Big|\ 
\begin{bmatrix}
      \tilde m \\
      \tilde u
\end{bmatrix}
=
\lambda A
\begin{bmatrix}
      \tilde m \\
      \tilde u
\end{bmatrix}
 \mbox{ for some }
\lambda \in [0,1]
\right\}.
\end{align*}
We want to prove that $S$ is bounded. When $\lambda=0$, we have $(\tilde m,\tilde u)\equiv (0,0)$. When 
 $\lambda\in (0,1]$, we proceed as follows.
 Assume that there exists $(\tilde m_\lambda, \tilde u_\lambda)\in S$. Because $(\frac{\tilde m_\lambda}{\lambda}, \frac{\tilde u_\lambda}{\lambda})$ is a pair of the minimizer to \eqref{vpt1} and the weak solution to \eqref{bft1},  
we have, for all
 $w\in \Aa_0$ and $v\in \Bb_0$,
 that%
\begin{align*}
&\int_0^T\!\!\!\int_{\Tt^d} \Big[\epsilon \bigg(\tilde m_\lambda(\lambda w - \tilde m_\lambda)+
\sum_{|\beta|=2k}\partial_{t,x}^\beta \tilde m_\lambda\partial_{t,x}^\beta(\lambda w- \tilde m_\lambda) \bigg) + 
\lambda f_1[\tilde m_\lambda, \tilde u_\lambda](\lambda w-\tilde m_\lambda)\Big]\,\dx\dt
\geq 0,\\
&\int_0^T\!\!\!\int_{\Tt^d}\Big[ \tilde u_\lambda v + \sum_{|\beta|=2k}\partial_{t,x}^\beta \tilde u_\lambda \partial_{t,x}^\beta v\Big]\,\dx\dt
=
\lambda\int_0^T\int_{\Tt^d}f_2[\tilde m_\lambda,\tilde u_\lambda]v
\,\dx\dt.
\end{align*}
Choosing $w=0$ and $v=\tilde u_\lambda$ in the preceding
inequalities,
and  adding the resulting estimates, we get
\begin{align*}
&\lambda\int_0^T\!\!\!\int_{\Tt^d}(m_\lambda g(m_\lambda)+(m_\lambda+m_0) |D
u_\lambda|^\alpha)\,\dx\dt\\
&\qquad+
\epsilon\int_0^T\!\!\!\int_{\Tt^d} \Big[\tilde m_\lambda^2 + \sum_{|\beta|=2k}(\partial_{t,x}^\beta \tilde m_\lambda)^2 + \tilde u_\lambda^2 + \sum_{|\beta|=2k}(\partial_{t,x}^\beta \tilde u_\lambda)^2\Big]\,\dx\dt\leq C,
\end{align*}
where $C$ is independent of $\lambda$.
By the Gagliardo--Nirenberg interpolation inequality, $(\tilde m_\lambda, \tilde u_\lambda)$ is  bounded in $H^{2k}(\Omega_T)\times H^{2k}(\Omega_T)$,
uniformly with respect to $\lambda$.
Therefore, $S$ is bounded in $\widehat \Aa\times \widehat \Bb$.
Hence, we can apply Theorem~\ref{Sch4-1} to $A$ and, thus, we construct a fixed point, $(\tilde m^\ast,\tilde u^\ast)\in \widetilde \Aa\times \widetilde \Bb$. Therefore $(m^\ast,u^\ast)=(\tilde m^\ast,\tilde u^\ast)+(m_0,u_T)$ solves
 \eqref{wsmt1}--\eqref{wsut1}.

\subsection{Weak solutions}

Let $(m_\epsilon,u_\epsilon)$ solve \eqref{wsmt1}--\eqref{wsut1}. 
We first establish some estimates for $(m_\epsilon, u_\epsilon)$ that are uniform in $\epsilon$.  

By Assumption~\ref{wcml1} and \eqref{apritr1}, there is $m\in L^1(\Omega_T)$ such that, extracting a subsequence if necessary, $m_\epsilon\rightharpoonup m$ in $L^1(\Omega_T)$. For $t\in (0,T)$, define
\begin{align*}
\left\langle u_\epsilon\right\rangle (t)=
\int_{\Tt^d}u_\epsilon(t,x)\,\dx.
\end{align*}
Setting $\overline u_\epsilon = u_\epsilon - \left\langle u_\epsilon\right\rangle $ and combining the Poincar\'{e}--Wirtinger inequality with \eqref{apritr1}, we obtain 
\begin{align*}
\|\overline u_\epsilon\|_{L^\alpha(\Omega_T)}^\alpha
\leq C\|D \overline u_\epsilon\|_{L^\alpha(\Omega_T)}^\alpha
\leq C
\end{align*}
for some  constant, $C>0$, uniform with respect to $\epsilon$. Thus, there is $\overline u\in L^\alpha([0,T]; W^{1,\alpha}(\Tt^d))$ such that, up to a subsequence, 
$\overline u_\epsilon\rightharpoonup \overline u$ in $L^\alpha([0,T]; W^{1,\alpha}(\Tt^d))$.
In particular, using \eqref{apritr1} once more, we have
\begin{align}\label{eq:comptd}
m_\epsilon \rightharpoonup m \mbox{ in } L^1(\Omega_T),\ 
\overline u_\epsilon \rightharpoonup \overline u \mbox{ in } W^{1,\alpha}(\Omega_T),\
\epsilon m_\epsilon \to 0 \mbox{ in }H^{2k}(\Omega_T), \text{ and } 
\epsilon u_\epsilon \to 0 \mbox{ in }H^{2k}(\Omega_T).
\end{align}

Let $\Aa^\ast$ be the subset of $\Aa$ (see \eqref{eq:defsets}) given by
\begin{align*}
\Aa^\ast
=
\left\{w\in \Aa\ \big|\ \int_{\Tt^d}w \,\dx=1 \ \mbox{for a.e. } t\in[0,T]\right\}.
\end{align*}
As before, for $\epsilon\in (0,1)$ and $(m,u)$ sufficiently smooth, we define two operators, $F$ and $F_\epsilon$, associated with the original time-dependent MFG and the regularized one, respectively, by
setting
\begin{align*}
F
  \begin{bmatrix}
      m \\
      u
  \end{bmatrix}
&=
  \begin{bmatrix}
      u_t+ \Delta u -H(Du) + g(m) + V \\
      m_t -\Delta m -\div\big(m D_p H(Du)\big) 
  \end{bmatrix},\\
F_\epsi\begin{bmatrix}
      m \\
      u
  \end{bmatrix}
  &= \begin{bmatrix}
       u_t+ \Delta u -H(Du) + g(m)+V+ \epsilon(m+\Delta^{2k}m ) \\ m_t -\Delta m -\div\big(m D_p H(Du)\big)+ \epsilon(u+ \Delta^{2k} u) 
      \end{bmatrix}. 
\end{align*}
Under  Assumptions~\ref{a0}--\ref{grg5-1} (adapted to the time-dependent case), and having in mind Remark~\ref{rmk:notation}, it can be checked that $F$ and $F_\epsilon$ are monotone over $\Aa^\ast \times \Bb$. 
Because $(m_\epsilon, u_\epsilon)$ satisfies \eqref{wsmt1}--\eqref{wsut1}, for all $(w,v)\in\Aa\times \Bb$, we have
\begin{align*}
\left\langle 
F_\epsilon
  \begin{bmatrix}
      m_\epsilon\\
      u_\epsilon
  \end{bmatrix}
,
  \begin{bmatrix}
      m_\epsilon \\
      u_\epsilon
  \end{bmatrix}
  -
  \begin{bmatrix}
      w \\
      v
  \end{bmatrix}
\right\rangle 
\leq 0.
\end{align*}
Fix $(w,v)\in \Aa^\ast \times \Bb$. 
Then, it can be checked that (see 
 Lemma~7.5 in \cite{FeGoTa21})
\begin{align*}
\left\langle 
F_\epsilon
  \begin{bmatrix}
      w\\
      v
  \end{bmatrix}
,
  \begin{bmatrix}
      m_\epsilon \\
      u_\epsilon
  \end{bmatrix}
  -
  \begin{bmatrix}
      w \\
      v
  \end{bmatrix}
\right\rangle 
=
\left\langle 
F_\epsilon
  \begin{bmatrix}
      w\\
      v
  \end{bmatrix}
,
  \begin{bmatrix}
      m_\epsilon \\
      \overline u_\epsilon
  \end{bmatrix}
  -
  \begin{bmatrix}
      w \\
      v
  \end{bmatrix}
\right\rangle ,
\end{align*}
where, we recall, $\overline u_\epsilon = u_\epsilon - \left\langle u_\epsilon\right\rangle $.
Hence, as in \eqref{eq5-1-1}, for $(w,v)\in \Aa^\ast\times \Bb$, we have that
\begin{equation}
\begin{aligned}\label{eq:contdr}
0\geq
\left\langle  
F_\epsilon
  \begin{bmatrix}
      m_\epsilon\\
      u_\epsilon
  \end{bmatrix}
,
  \begin{bmatrix}
      m_\epsilon \\
      u_\epsilon
  \end{bmatrix}
  -
  \begin{bmatrix}
      w \\
      v
  \end{bmatrix}
\right\rangle 
\geq
\left\langle  
F_\epsilon
  \begin{bmatrix}
      w\\
      v
  \end{bmatrix}
,
  \begin{bmatrix}
      m_\epsilon \\
      u_\epsilon
  \end{bmatrix}
  -
  \begin{bmatrix}
      w \\
      v
  \end{bmatrix}
\right\rangle 
&=
\left\langle  
F_\epsilon
  \begin{bmatrix}
      w\\
      v
  \end{bmatrix}
,
  \begin{bmatrix}
      m_\epsilon \\
      \overline u_\epsilon
  \end{bmatrix}
  -
  \begin{bmatrix}
      w \\
      v
  \end{bmatrix}
\right\rangle \\
&=
\left\langle  
F
  \begin{bmatrix}
      w\\
      v
  \end{bmatrix}
,
  \begin{bmatrix}
      m_\epsilon \\
     \overline u_\epsilon
  \end{bmatrix}
  -
  \begin{bmatrix}
      w \\
      v
  \end{bmatrix}
\right\rangle 
+
h_\epsilon,
\end{aligned}
\end{equation}
where
\begin{align*}
h_\epsilon&=
\epsilon\int_0^T\int_{\Tt^d}\Big[
w(w-m_\epsilon)+\sum_{|\beta|=2k}\partial_{t,x}^\beta w\partial_{t,x}^\beta(w-m_\epsilon)
+
v(v-\overline u_\epsilon)+\sum_{|\beta|=2k}\partial_{t,x}^\beta v\partial_{t,x}^\beta(v-\overline u_\epsilon)
\Big]\,\dx\dt.
\end{align*}
By   \eqref{eq:comptd} and \eqref{eq:contdr},  $\lim_{\epsilon\to0}h_\epsilon=0$ and  
\begin{equation*}
0\geq
\lim_{\epsilon\to0}\left(
\left\langle  
F
  \begin{bmatrix}
      w\\
      v
  \end{bmatrix}
,
  \begin{bmatrix}
      m_\epsilon \\
     \overline u_\epsilon
  \end{bmatrix}
  -
  \begin{bmatrix}
      w \\
      v
  \end{bmatrix}
\right\rangle 
+
h_\epsilon
\right)
=
\left\langle  
F
  \begin{bmatrix}
      w\\
      v
  \end{bmatrix}
,
  \begin{bmatrix}
      m \\
     \overline u
  \end{bmatrix}
  -
  \begin{bmatrix}
      w \\
      v
  \end{bmatrix}
\right\rangle , 
\end{equation*}
which is a weak solution to \eqref{ext1} in the spirit of Remark \ref{412}. To get a weak solution like in Theorem~\ref{mt1}, where the test functions 
belong to $\Aa\times \Bb$, we would have to prove the convergence of $u_\epsilon$. The authors leave this matter as an open problem for the interested
reader.

 \bibliographystyle{plain}

\def\cprime{$'$}

\end{document}